\documentclass{article}
\usepackage{amsmath, amssymb, amsthm}
\usepackage[margin=1.25in]{geometry}
\usepackage[utf8]{inputenc}
\usepackage{enumitem}
\usepackage{tikz-cd}
\usepackage{xfrac}
\usepackage{dsfont}
\usepackage{dsfont}
\usepackage{cancel}
\usepackage{csquotes}
\usepackage{mathrsfs}
\usepackage{calc} % Specifically to do the calculations in the figures stuff
\usepackage{tensor}
\usepackage[pdfencoding=unicode,pdfusetitle]{hyperref}
\hypersetup{
    colorlinks=true,
    linkcolor=blue,
    filecolor=purple,      
    urlcolor=[rgb]{0 0 .6},
	 psdextra,
}
\usepackage[backend=biber,style=alphabetic,citestyle=alphabetic,url=false,isbn=false,maxnames=5,maxalphanames=5]{biblatex}
\usepackage{refcount}

\SetLabelAlign{Center}{\hfil#1\hfil}
\SetLabelAlign{CenterWithParen}{\hfil(\makebox[1.0em]{#1})\hfil}

\newcommand{\1}{\mathds{1}}
\newcommand{\s}{\mathcal}

\newcommand{\id}{\text{id}}
\newcommand{\im}{\text{im}}
\newcommand{\rk}{\text{rk}}
\newcommand{\ob}{\text{ob}}
\renewcommand{\S}{\mathbb N\mathcal O}
\renewcommand{\char}{\text{char}}
\newcommand{\Set}{\mathrm{Set}}

\newcommand{\cMon}{\mathrm{cMon}}

\newcommand{\Hom}{\mathrm{Hom}}

\newcommand{\FPdim}{\mathrm{FPdim}}

\newcommand{\Gal}{\mathrm{Gal}}

\newcommand{\End}{\mathrm{End}}
\renewcommand{\Vec}{\mathrm{Vec}}
\newcommand{\Rep}{\mathrm{Rep}}
\newcommand{\Rex}{\mathrm{Rex}}
\renewcommand{\O}{\mathcal{O}}
\newcommand{\Mod}{\text{-}\mathrm{Mod}}

\newcommand{\Bim}{\text{-}\mathrm{Bim}}

\renewcommand{\lim}{\mathop{\mathrm{lim}}}
\makeatletter
\newtheorem*{rep@theorem}{\rep@title}
\newcommand{\newreptheorem}[2]{%
\newenvironment{rep#1}[1]{%
 \def\rep@title{#2 \ref{##1}}%
 \begin{rep@theorem}}%
 {\end{rep@theorem}}}
\makeatother

\newtheorem{theorem}{Theorem}[section]
\newreptheorem{theorem}{Theorem}
\newtheorem{proposition}[theorem]{Proposition}
\newtheorem{corollary}[theorem]{Corollary}
\newtheorem{lemma}[theorem]{Lemma}

\theoremstyle{definition}
\newtheorem{definition}[theorem]{Definition}
\newtheorem{example}[theorem]{Example}

\newtheorem{remark}[theorem]{Remark}

\title{Fusion Categories over Non-Algebraically Closed Fields}
\date{}

\author{Sean Sanford\\\itshape\small Department of Mathematics, The Ohio State University}

\begin{document}

\maketitle

\begin{abstract}
    Several complications arise when attempting to work with fusion categories over arbitrary fields.  Here we describe some of the new phenomena that occur when the field is not algebraically closed, and we adapt tools such as the Frobenius-Perron dimension in order to accommodate these new effects.
\end{abstract}

\tableofcontents

\section{Introduction}

% What is the Field?
Fusion categories are categories that behave like rings.  Their hom-sets are vector spaces over a field, which allows for a direct sum $\oplus$ operation, and they are endowed with a product functor $\otimes$ which is associative and distributes over $\oplus$ up to isomorphism.  They are semisimple, so the tensor product of two simple objects can be written as a sum $X\otimes Y=\bigoplus_ZN_{X,Y}^Z\cdot Z$ of simple objects $Z$ with certain multiplicities $N_{X,Y}^Z$.  Since these multiplicities are nonnegative integers, the theory of fusion categories takes on a distinctly combinatorial flavor, in contrast to other category-theoretic disciplines.

% Why is it interesting?
Although the theory of fusion categories is an algebraic discipline, it has strong links to topology and mathematical physics.  Topological quantum field theories (TQFTs) are functors that assign numerical invariants to closed smooth manifolds.  Fusion categories are the algebraic input to the Turaev-Viro state sum construction which produces examples of (2+1)D TQFTs, see \cite{turaevStateSum}.  Fusion categories that are modular and unitary can be used to classify (2+1)D gapped topological phases of matter, see \cite{wenZooQuantumtopologicalPhases2017}.

The goal of this paper is to initiate the extension of the theory of fusion categories to the realm of non\textendash algebraically closed (nAC) fields.  When working over nAC fields, constructions such as the Frobenius-Perron dimension ($\FPdim$) and the Drinfeld center behave differently, and we will give explanations and formulas to make these differences precise.  Schur's lemma establishes that for any simple object $X$, $\End(X)$ must be a division algebra. Without algebraic closure (AC) of the base field, nontrivial division algebras can occur, and this is the root cause of the majority of new phenomena in the nAC setting.

Our first key result is the following

\begin{reptheorem}{Thm:NewRegular}
    Let $\mathcal C$ be fusion over $\mathbb K$ with endomorphism field $\mathbb E=\End(\1_{\mathcal C})$.  Up to a positive scalar factor, the regular element $\s R$ of $\mathcal C$ is given by
    \[\s R\;=\;\sum_{X\in\s O(\mathcal C)}\frac{\FPdim(X)}{\dim_{\mathbb E}\End(X)}\cdot [X]\,.\]
\end{reptheorem}

In analogy with the AC setting, we define $\FPdim(\mathcal C):=\FPdim(\mathcal R)$.  Unlike the classical case, it is not immediately clear whether or not this new $\mathrm{FPdim}$ is an algebraic integer.  Nevertheless, we are able to show that this desirable property remains true in the nAC setting.

\begin{reptheorem}{Thm:FPdimRIsFPdimC}
    For every fusion category $\s C$ over $\mathbb K$, there exists an object $R$ in $\s C$, such that \linebreak $\FPdim(R)\;=\;\FPdim(\s C)$.  In particular, $\FPdim(\s C)$ is necessarily an algebraic integer.
\end{reptheorem}

One important feature is that Galois theory becomes relevant to the theory of fusion categories in the nAC setting.  This occurs because the algebra $\End(\1_{\mathcal C})$ is generally a field extension, of degree say $d_{\mathcal C}$, over the base field.  Because of this, we introduce a notion of \emph{Galois nontriviality} that explains how the simple objects interact with the Galois group.  This new Galois nontriviality causes the Drinfeld center to be smaller, as measured by $\FPdim$, than it would normally be in the AC setting.

\begin{reptheorem}{Thm:FPdimOfZ(C)}
    Suppose both $\mathcal C$ and $\mathcal Z(C)$ are fusion over $\mathbb K$, and let $F:\mathcal Z(\mathcal C)\to\mathcal C$ be the forgetful functor. The Frobenius-Perron dimension of $\s Z(\s C)$ satisfies
    \[\FPdim\big(\s Z(\s C)\big)\;=\;\left(\frac{d_{\s Z(\s C)}}{d_{\s C}}\right)\FPdim\Big(\im(F)\Big)\FPdim(\s C)\leq\FPdim(\mathcal C)^2\,.\]
    Moreover, equality holds if and only if all objects of $\mathcal C$ are Galois trivial. 
\end{reptheorem}

The presence of Galois theory should not be surprising.  It was already observed in \cite[Prop 4.9]{MR2677836} that the Brauer group of the base field is relevant to the extension theory of fusion categories over nAC fields, and in future papers we will explain how the Morita theory of fusion categories in the nAC setting is connected to higher Galois cohomology.

This paper was originally part of the author's PhD thesis \cite{sanfordThesis}.  He would like to thank his advisor Julia Plavnik for her advice and patience, and also Noah Snyder for many useful conversations.  Many thanks are due to David Penneys for his edits and suggestions for the article version of the current draft.  This research was partially supported by the National Science Foundation under Grant No. DMS-2000093.

\section{Preliminaries}

Before launching into new results, it will be necessary to give a sense of the ways in which our categories differ from the classical notion of fusion categories in the AC setting.  We expect that the reader is familiar with the material from \cite[Chapters 1, 2, 3, and 4]{MR3242743}.

\begin{definition}
    A simple object in an additive category $\s C$ is a nonzero object $X$ such that for any object $Y$ and any monomorphism $i:Y\hookrightarrow X$, either $i$ is an isomorphism or $Y=0$.
\end{definition}

\begin{example}
    In the category $\Rep_{\mathbb K}(G)$ of finite dimensional $\mathbb K$-linear representations of a finite group $G$, the simple objects are precisely the irreducible representations.
\end{example}

The following observation about simple objects was first proved in the special case of irreducible representations by Issai Schur, and has since come to be known as Schur's Lemma.  The form that appears below is a broad generalization of the original, though its proof is nearly identical.

\begin{lemma}[Schur's Lemma]\label{Schur'sLemma}
    Let $X$ and $Y$ be simple objects in a $\mathbb K$-linear abelian category.  The vector space $\Hom(X,Y)=0$ unless $X\cong Y$, and $\End(X)$ is a division algebra over $\mathbb K$.
\end{lemma}

\begin{example}\label{RepQ8}
    For the quaternion group $Q_8$, the category $\Rep_{\mathbb R}(Q_8)$ has 5 simple objects up to isomorphism.  These correspond to the irreducible representations, four of which are one dimensional and one that is four dimensional.  Lemma \ref{Schur'sLemma} tells us that the spaces of endomorphisms of each irreducible representation are necessarily division algebras over $\mathbb R$.  The one dimensional irreducible representations produce algebras that are isomorphic to $\mathbb R$, while the four dimensional irreducible representation produces the quaternion algebra $\mathbb H$.
\end{example}

\begin{definition}
    Let $\s C$ be a $\mathbb K$-linear category.  A simple object $X$ in $\s C$ that satisfies $\End(X)\cong\mathbb K$ is said to be split, or split-simple.
\end{definition}

\begin{example}
    Over any AC field, the only finite dimensional division algebra is the field itself, so all simple objects are automatically split.
\end{example}

Lemma \ref{Schur'sLemma} and Example \ref{RepQ8} are the jumping off point for this article.  The existence of simple objects that are not split is a phenomenon that is rarely considered in the current literature.  It is common for authors to assume all simples are split (as in \cite{bakalovLectures} and \cite{TAMBARA1998692}), or more restrictively that $\mathbb K$ is algebraically closed (as in \cite{MR3242743}).

The 2-category of all finite $\mathbb K$-linear abelian categories admits a monoidal structure known as the Deligne tensor product $\boxtimes_{\mathbb K}$.  Roughly speaking, there is a na\"ive tensor product $\mathcal C\otimes_{\mathbb K}\mathcal D$ where
\[\Hom_{\mathcal C\otimes_{\mathbb K} \mathcal D}(X\otimes_{\mathbb K} U,Y\otimes_{\mathbb K} V):=\Hom_{\mathcal C}(X,Y)\otimes_{\mathbb K}\Hom_{\mathcal D}(U,V)\,.\]
Unfortunately, this na\"ive tensor product category may not be abelian in general, and so the Deligne product $\mathcal C\boxtimes_{\mathbb K}\mathcal D$ is the formal completion of $\mathcal C\otimes_{\mathbb K}\mathcal D$ under finite colimits (Skip ahead to Definition \ref{Def:DeligneTensor} for a precise definition and relevant universal properties).

\begin{example}\label{Eg:DeligneTensorOfSimplesNotSimple}
    Consider the category $\Vec_{\mathbb C}$ as a $\mathbb R$-linear category.  In order to understand $\s B:=\Vec_{\mathbb C}\boxtimes_{\mathbb R}\Vec_{\mathbb C}$, it is necessary to determine the idempotents in the algebra
    \[\End(\mathbb C)\otimes_{\mathbb R}\End(\mathbb C)\;\cong\;\mathbb C\otimes_{\mathbb R}\mathbb C\,.\]
    A short computation produces two nontrivial idempotents:
    \begin{align*}
        p&=\tfrac12(1\otimes1-i\otimes i)\\
        q&=\tfrac12(1\otimes1+i\otimes i)\,.
    \end{align*}
    It follows that the simple objects of $\s B$ are $(\mathbb C,\mathbb C,p)$ and $(\mathbb C,\mathbb C,q)$ (see Definition \ref{Def:DeligneTensor} for notation conventions).
\end{example}

Whenever $X$ and $Y$ are split simple objects, it follows that $X\boxtimes_{\mathbb K}Y$ is a split simple object in the Deligne product category.  Thus Example \ref{Eg:DeligneTensorOfSimplesNotSimple} is a significant departure from the behavior of Deligne products in the AC setting.  When working over $\mathbb K=\mathbb R$, the following example gives a complete description of the ways in which Deligne products of simple objects can decompose.

\begin{example}\label{Eg:RealDeligneTensors}
    A famous theorem of Frobenius \cite{frobeniusUberLineare} states that the only finite dimensional division algebras over $\mathbb R$ are $\mathbb R$ itself, $\mathbb C$, and $\mathbb H$ (the quaternion algebra).  Thus for any simple object $X$ in an $\mathbb R$-linear abelian category, let us say that $X$ is real, complex, or quaternionic whenever $\End(X)$ is isomorphic to $\mathbb R$, $\mathbb C$, or $\mathbb H$ respectively.  The tensor products of these algebras over $\mathbb R$ are shown in the chart below.
    \begin{equation*}
        \begin{array}{c|c|c|c|}
            \downarrow A\otimes_{\mathbb R}B \rightarrow & \mathbb R & \mathbb C & \mathbb H \\\hline
            \mathbb R & \mathbb R & \mathbb C & \mathbb H  \\\hline
            \mathbb C & \mathbb C & \mathbb C\oplus\mathbb C & M_2(\mathbb C) \\\hline
            \mathbb H & \mathbb H & M_2(\mathbb C) & M_4(\mathbb R)  \\\hline
        \end{array}
    \end{equation*}

    From the above table, we can deduce that when both $X$ and $Y$ are complex, $X\boxtimes Y$ decomposes into two non-isomorphic complex simple objects (cf. Example \ref{Eg:DeligneTensorOfSimplesNotSimple}).  When both $X$ and $Y$ are quaternionic, $X\boxtimes Y$ decomposes into four isomorphic copies of a unique real simple object.  Finally when $X$ is complex and $Y$ is quaternionic, $X\boxtimes Y$ decomposes into two isomorphic copies of a unique complex simple object.
\end{example}

\begin{definition}\label{Def:mFus}
    A multifusion category over $\mathbb K$ is a rigid monoidal category $\s C$ that is separable (and hence abelian, see Definition \ref{Def:Separable}) over $\mathbb K$ and whose product $\otimes$ is $\mathbb K$-linear in both arguments.  If in addition $\1$ is simple, we say that $\s C$ is a fusion category.
\end{definition}

\begin{example}
    The category $\Rep_{\mathbb K}(G)$ is a fusion category whenever $\char(\mathbb K)\nmid |G|$
\end{example}

\begin{example}
    Let $\s M_n(\Vec_{\mathbb K})$ be the semisimple category generated by simple objects $E_{i,j}$, with $\End(E_{i,j})\cong\mathbb K$ for all $i,j\in\{1,\cdots,n\}$.  The tensor product is given by
    \[E_{i,j}\otimes E_{k,\ell}:\;=\;\begin{cases}
    0 & \text{ if }j\neq k\\
    E_{i,\ell} & \text{ if }j=k
    \end{cases}\,.\]
    The associators and unitors are identities.  This category is known as the category of $n\times n$ matrices over $\Vec_{\mathbb K}$.  From the definition of the tensor product, it is easy to see that
    \[\1\;=\;\bigoplus_{i=1}^nE_{i,i}\,\]
    which is not simple.  This category is the prototypical example of a multifusion category.
\end{example}

\begin{example}\label{VecGIsFusion}
    The category $\Vec_{\mathbb K}(G)$ of (finite dimensional) $G$-graded $\mathbb K$-vector spaces admits a monoidal structure, where
    \[(V\otimes W)_g:=\bigoplus_{hk=g}V_h\otimes_{\mathbb K}W_k\,.\]
    This product is $\mathbb K$-linear and exact in both variables.  The unit in this category is the vector space where $\1_1=\mathbb K$ and $\1_g=0$ for any other $g\in G$.  The object $\1$ is simple for dimension reasons.  This category is equivalent to $A\Mod$, where $A=\Hom_{\Set}(G,\mathbb K)$ in $\Vec_{\mathbb K}$, which is separable.  Thus $\Vec_{\mathbb K}(G)$ is a fusion category.
\end{example}

\begin{example}\label{Eg:CCbim}
    Regard $\mathbb C$ as an algebra object in the category $\Vec_{\mathbb R}$.  Then the category \linebreak $(\mathbb C,\mathbb C)\Bim$ of bimodules for $\mathbb C$ is fusion over $\mathbb R$.  It contains two simple objects, the trivial bimodule $\1$ and the conjugating bimodule $C$.  As a left $\mathbb C$-vector space, both objects are isomorphic to $\mathbb C$.  The left and right actions of $\mathbb C$ on $\1$ are the same, while the left and right actions on $C$ differ by complex conjugation.
\end{example}

\begin{example}\label{Eg:FusionDependsOnField}
    Whether or not a category is fusion generally depends on what base field is being used.  Consider an inseparable extension $\mathbb L/\mathbb K$.  The category $\Vec_{\mathbb L}$ is fusion as a category over $\mathbb L$, but not as a category over $\mathbb K$.  This is due to the fact that $\mathbb L$ is not separable as a $\mathbb K$-algebra.
\end{example}

It is important at this point to emphasize the ways in which our Definition \ref{Def:mFus} differs from the classical definition.  As Example \ref{Eg:FusionDependsOnField} points out, fusion is always with respect to some background field.  When multiple fields are involved in a scenario, it should always be stressed which field the categories are fusion with respect to.

The classical definition of multifusion relies on semisimplicity as opposed to separability.  The motivation for the new definition comes from Example \ref{Eg:SemXSem=NonSem}.  Since all semisimple algebras over perfect fields (e.g. algebraically closed fields) are separable, the difference in our choice of generalization only appears when working over non-perfect fields.  Non-perfect fields are necessarily infinite and of positive characteristic, such as $\mathbb F_p(x)$.

In the AC setting, there is no difference between the statement $\End(\1)\cong\mathbb K$ and the statement that $\1$ is simple.  We choose to keep the statement that $\1$ is simple, because this allows for more interesting examples.  In light of this choice, the Eckmann-Hilton argument and Lemma \ref{Schur'sLemma} combine to imply that $\End(\1)$ must be a finite degree field extension of $\mathbb K$.  We have already seen this happen in Example \ref{Eg:CCbim}, and we wish to include categories like this.  One unavoidable consequence of this choice is that the tensor product of fusion categories is typically multifusion and not fusion, because when the unit is not split, $\1\boxtimes_{\mathbb K}\1$ typically has more than one simple summand (see Definition \ref{Def:DeligneTensor} for details).

Since we will allow $\End(\1)$ to be a nontrivial field extension of $\mathbb K$, this field will play an important role in the nAC setting.  For this reason, we establish some new language to refer to $\End(\1)$ and its properties.

\begin{definition}\label{EndoField&Degree}
    To any fusion category $\s C$ over $\mathbb K$, the endomorphism field (of $\s C$) is the field
    \[\mathbb E\;:=\;\mathbb E(\s C)\;=\;\End(\1)\,.\]
    Elements of the endomorphism field behave similarly to scalars, but behave differently in a few key ways, and for this reason will sometimes be referred to as \emph{pseudo-scalars}.  The endomorphism degree, or just degree for short, is the number
    \[d_{\s C}:=[\mathbb E:\mathbb K]\]
\end{definition}

\begin{remark}
    It should be stressed that often a fusion category $\s C$ over $\mathbb K$ is not fusion over its endomorphism field $\mathbb E$.  The failure to be fusion has to do with the requirement that $\otimes$ be bilinear over the chosen base field, as the following example illustrates.
\end{remark}

\begin{example}\label{Eg:CCbimContd}
    Let $\mathcal C=(\mathbb C,\mathbb C)\Bim$ be the category from Example \ref{Eg:CCbim}.  This category has degree $d_{\mathcal C}=2$ as a fusion category over $\mathbb R$, because $\mathbb E=\mathbb C$.  Although both the simple objects have $\mathbb C$ as their endomorphism algebra, this category is not fusion over $\mathbb C$, because the monoidal product is not $\mathbb C$-bilinear on morphisms.  This can be seen by observing the fact that
    \[(\lambda\cdot\id_C)\otimes_{\mathbb C}(\nu\cdot\id_C)\;=\;\lambda\overline{\nu}\cdot(\id_C\otimes_{\mathbb C}\id_C)\,.\]
\end{example}

This phenomenon of being `almost linear' over $\mathbb E(\s C)$ but failing due to Galois conjugation is something that is new in the nAC case, and it has significant consequences.  In order to describe it properly, we will need a few definitions.

\pagebreak

\begin{definition}\label{Def:LREmbeddings}
    Let $X$ be an object in a fusion category, and let $e:\1\to\1$ be an element of $\mathbb E(\s C)$.  There are two canonical embeddings $\lambda_X,\rho_X:\End(\1)\hookrightarrow\End(X)$ defined by the following diagrams.
    % https://q.uiver.app/?q=WzAsOCxbMSwwLCJYIl0sWzAsMSwiXFwxXFxvdGltZXMgWCJdLFswLDIsIlxcMVxcb3RpbWVzIFgiXSxbMSwzLCJYIl0sWzIsMCwiWCJdLFsyLDMsIlgiXSxbMywxLCJYXFxvdGltZXNcXDEiXSxbMywyLCJYXFxvdGltZXNcXDEiXSxbMCwzLCJcXGxhbWJkYV9YKGUpIiwyXSxbNCw1LCJcXHJob19YKGUpIl0sWzAsMSwiXFxlbGxfWF57LTF9IiwyXSxbMSwyLCJlXFxvdGltZXNcXGlkX1giLDJdLFsyLDMsIlxcZWxsX1giLDJdLFs0LDYsInJfWF57LTF9Il0sWzYsNywiXFxpZF9YXFxvdGltZXMgZSJdLFs3LDUsInJfWCJdXQ==
    \[\begin{tikzcd}
    	& X & X \\
    	{\1\otimes X} &&& X\otimes\1 \\
    	{\1\otimes X} &&& X\otimes\1 \\
    	& X & X
    	\arrow["{\lambda_X(e)}"', from=1-2, to=4-2]
    	\arrow["{\rho_X(e)}", from=1-3, to=4-3]
    	\arrow["{\ell_X^{-1}}"', from=1-2, to=2-1]
    	\arrow["{e\otimes\id_X}"', from=2-1, to=3-1]
    	\arrow["{\ell_X}"', from=3-1, to=4-2]
    	\arrow["{r_X^{-1}}", from=1-3, to=2-4]
    	\arrow["{\id_X\otimes e}", from=2-4, to=3-4]
    	\arrow["{r_X}", from=3-4, to=4-3]
    \end{tikzcd}\]
    We will refer to $\lambda_X$ and $\rho_X$ as the left and right embeddings respectively.
\end{definition}

\begin{definition}\label{Def:GaloisNontrivial}
    An object $X$ in a fusion category over $\mathbb K$ is called Galois trivial if $\lambda_X(e)=\rho_X(e)$ for all $e\in\mathbb E(\s C)$, and $X$ is called Galois nontrivial otherwise.
\end{definition}

\begin{remark}
    The existence of Galois nontrivial objects is only possible when $d_{\s C}>1$, because otherwise $\mathbb E=\mathbb K$ and $\mathbb K$-bilinearity allows all scalars to be factored out.  In general, if all objects are Galois trivial, then the category is fusion over $\mathbb E$.
\end{remark}

The terminology Galois trivial is meant to indicate that the given object does not interact with the field extension $\mathbb E(\mathcal C)/\mathbb K$.  When an object $X$ is Galois nontrivial, this means that passing pseudo-scalars across the object $X$ has a nontrivial interaction with the Galois group $\Gal(\mathbb E/\mathbb K)$, when $\mathbb E/\mathbb K$ is Galois.  In the event that $\mathbb E/\mathbb K$ is not normal, Galois nontriviality of $X$ indicates that there is a nontrivial interaction with the Galois group of the Galois closure instead.  The following two examples expand on these two cases respectively.

\begin{example}
    It turns out that Example \ref{Eg:CCbim} generalizes to any Galois extension.  Let $\mathbb L/\mathbb K$ be a Galois extension of degree $n$, and regard $\mathbb L$ as an algebra in $\Vec_{\mathbb K}$.  The category $(\mathbb L,\mathbb L)\Bim$ has exactly $n$ simple objects.  It can be shown that for each element $g\in\Gal(\mathbb L/\mathbb K)$, there is a unique simple object $\mathbb L_g$.  This simple object is determined by the condition that, for every $e\in\End(\1)\cong\mathbb L$,
    \[\rho_{\mathbb L_g}(e)\;=\;\lambda_{\mathbb L_g}\big(g(e)\big)\,.\]
    Thus when $g\neq1$, the object $\mathbb L_g$ is Galois nontrivial by Definition \ref{Def:GaloisNontrivial}.
    By taking the tensor product of bimodules and comparing the left and right embeddings, one quickly arrives at the following formula for the tensor product
    \[\mathbb L_g\otimes_{\mathbb L}\mathbb L_h\cong\mathbb L_{gh}\,.\]
    From this it follows that all simple objects are invertible, and have endomorphism algebras isomorphic to $\mathbb E(\s C)=\mathbb L$.
\end{example}

So far we have seen categories where Galois nontriviality amounts to the left and right embeddings differing by a Galois transformation.  This makes it so that pseudo-scalars can be passed from right to left across the symbol $\otimes$ by applying elements of the Galois group.  In this next example we will see an object with the even less convenient property $\im(\lambda_X)\neq\im(\rho_X)$.  When an object is Galois nontrivial in this way, there are pseudo-scalars that cannot be passed across the symbol $\otimes$ at all.

\begin{example}\label{NonNormalBim}
    Let $\mathbb J:=\mathbb Q(\sqrt[3]{2}\,)$. The field extension $\mathbb J/\mathbb Q$ is separable, but not normal.  By using $\omega$ and $\omega^2$ as the nontrivial third roots of unity, the other roots of $x^3-2$ are $\omega\sqrt[3]{2}$ and $\omega^2\sqrt[3]{2}$.  The category $(\mathbb J,\mathbb J)\Bim$ has two simple objects $\1:=\mathbb J$ and the splitting field $X:=\mathbb J(\omega)$.  As in the previous examples, the left (resp. right) embeddings come from the left(right) module actions.  These actions are determined for $X$ by the equations
    \begin{gather*}
        \lambda_X\left(\sqrt[3]{2}\,\right)=\omega\sqrt[3]{2}\hspace{3mm}\text{ and }\hspace{3mm}
        \rho_X\left(\sqrt[3]{2}\,\right)=\omega^2\sqrt[3]{2}\,,
    \end{gather*}
    and hence this object $X$ is Galois nontrivial by Definition \ref{Def:GaloisNontrivial}.
    Graphically, these formulas show that the act of passing the pseudo-scalar $\sqrt[3]{2}$ from right to left across $X$ results in a `residue' of $\omega$ that remains on the $X$ strand.
\end{example}

\section{Frobenius-Perron Dimensions}

    The algebraic objects that underlie the combinatorics of fusion categories are known as $\mathbb Z_+$-rings.  We partially follow the treatment given in \cite[Ch 3]{MR3242743}, though we are forced to make certain changes to the definition of a fusion ring.  These changes stem from the need to keep track of the dimensions of various division rings that appear in the nAC setting.  Since we are already making changes, we take this opportunity to recast the combinatorics of fusion categories in terms of a fundamental combinatorial object: subsets with multiplicity, also known as multisubsets.

    Consider the functor $\mathbb N:\Set\to\cMon$ that assigns to every set $S$, the free commutative monoid $\mathbb NS$ generated by $S$.  The elements of $\mathbb NS$ are finite formal $\mathbb N$-linear combinations of the elements in $S$.  We will denote by $\mathbb N\Set$ the full image of $\mathbb N$.  Note that a typical element of $\mathbb NS$ is of the form $a=\sum_{s\in S}n_ss$, and these are in bijection with finite multisubsets of $S$.  In the multisubset context, the coefficient $n_s\in\mathbb N$ is referred to as the multiplicity of $s$ in the multisubset $a$.

    \begin{definition}
        Every object $\mathbb NS$ in $\mathbb N\Set$ admits a partial ordering.  We say that $a\leq b$ if there exists some $c\in \mathbb NS$ such that $a+c=b$.  If $c\neq0$ and $a+c=b$, then we write $a<b$.  The intersection of two elements $a=\sum a_ss$ and $b=\sum b_ss$ is the element $a\cap b:=\sum \min(a_s,b_s)s$.  An element $b\in\mathbb NS$ is said to be simple if $0<a\leq b$ implies $a=b$.  Note that the simple elements of $\mathbb NS$ are the singleton multisubsets $\{s\}=1s$ for $s\in S$.  If we identify $S$ with this collection of singleton multisubsets, then the simple elements are precisely the elements of $S\subset\mathbb NS$.
    \end{definition}

    The monoidal structure from $\Set$ induces a monoidal structure on $\mathbb N\Set$: $\mathbb NS\otimes\mathbb N T=\mathbb N(S\times T)$.  It is immediate that $\mathbb N\;=\;\mathbb N\{*\}$ is the monoidal unit, and that $\otimes=\otimes_{\mathbb N}$, the relative tensor product over the natural numbers.

    \begin{definition}
        For an algebra object $A$ in $\mathbb N\Set$, we establish the following notation
        \begin{itemize}
            \item $\O\;=\;\O(A)$ is the set of simple elements of $A$,
            \item $\rk(A)=|\mathcal O(A)|$ is the rank of $A$,
            \item $\O_0\;=\;\O_0(A)\;=\;\{b\in\O(A)\;|\;b\leq 1\}$,
            \item $ab=\sum_{c\in\O}N_{a,b}^cc$, these $N_{a,b}^c$ are called the fusion coefficients of $A$.
        \end{itemize}
    \end{definition}

    The following two propositions are general facts about algebra objects in $\mathbb N\Set$.

    \begin{proposition}[{\cite[cf. Prop 3.1.4]{MR3242743}}]\label{Prop:1LooksLikeThis}
        For any algebra $A$ in $\mathbb N\Set$, any elements $a,b\in\mathcal O_0$ (summands of $1$) satisfy the relation $ab=\delta_{a,b}a$.  Furthermore, it follows that $1\;=\;\sum_{b\in\O_0}b$.
    \end{proposition}

    \begin{proof}
        In general, we know that $1$ decomposes like $1\;=\;\sum_{b\in\s O_0}n_bb$, where all of the $n_b\geq1$.  This implies that for any $a\in\s O$,
        \begin{gather*}
            a=a1
            =a\left(\sum_{b\in\s O_0}n_bb\right)
            =\sum_{b\in\s O_0}n_bab\\
            =\sum_{\substack{{b\in\s O_0}\\{c\in\s O}}}n_bN_{a,b}^cc
            =\sum_{c\in\s O}\left(\sum_{b\in\s O_0}n_bN_{a,b}^c\right)c\,.
        \end{gather*}
        By equating coefficients, the above implies that $\delta_{a,c}=\sum_{b\in\s O_0}n_bN_{a,b}^c$.  All of the $N_{a,b}^c$ are natural numbers, so the only way for this last equation to be satisfied is if to each $a\in\s O$ there corresponds a unique $f(a)\in\s O_0$ that satisfies $N_{a,b}^c=\delta_{a,c}\delta_{b,f(a)}$.  By examining the decomposition of $1a$, the above argument can be repeated to reveal the existence of another function $g:\s O\to\s O_0$, where $N_{d,a}^c=\delta_{a,c}\delta_{d,g(a)}$.
        
        These observations combine to show that when $a\in\s O_0$,
        \[1\;=\;N_{a,f(a)}^a\;=\;\delta_{f(a),a}\delta_{a,g(f(a))}\,,\]
        which implies $f(a)=a=g(a)$.  In other words, $ab=\delta_{a,b}a$ for any $a,b\in\s O_0$, so the first claim is established.  With this new formula in hand, we can compute
        \begin{gather*}
            \sum_{b\in\s O_0}n_bb=1
            =1^2
            =\left(\sum_{b\in\s O_0}n_bb\right)^2\\
            =\sum_{a,b\in\s O_0}n_an_bab
            =\sum_{a,b\in\s O_0}n_an_b\delta_{a,b}a
            =\sum_{a,b\in\s O_0}n_a^2a\,.
        \end{gather*}
        Thus $n_a=n_a^2$ for every $a\in\s O_0$.  Since each $n_a\geq1$, it must be the case that $n_a=1$ for every $a$.
    \end{proof}

    \begin{proposition}\label{TechnicalProjLemma}
        If $p\in A$ is a idempotent, i.e. $p^2=p$, such that $1\leq p$, then $p=1$.
    \end{proposition}

    \begin{proof}
        Let $p=\sum_{a\in\s O}n_aa$.  The projection property shows us that
        \begin{gather*}
            \sum_{c\in\s O}n_cc\;=\;p\;=\;p^2\;=\;\sum_{a,b\in\s O}n_an_bN_{a,b}^cc\,.
        \end{gather*}
        By comparing coefficients, we arrive at the formula
        \begin{gather}
            n_c\;=\;\sum_{a,b\in\s O}n_an_bN_{a,b}^c\,.\label{Eqn:ProjCoeff}
        \end{gather}
        If $b\in\s O_0$, Proposition \ref{Prop:1LooksLikeThis} tells us that $b^2=b$.  Next, set $c=b\in\mathcal O_0$ in Equation \ref{Eqn:ProjCoeff}, and then isolate the $a=b$ term to find that
        \[n_b\;=\;n_b^2+B\;\geq\;n_b^2\;\geq\;n_b\,,\]
        where $B$ is simply the sum of all the other terms from Equation \ref{Eqn:ProjCoeff}.  If this $B$ is nonzero, then $n_b>n_b$, a contradiction.  Thus $n_b=n_b^2$ and this forces $n_b$ to be either $0$ or $1$.  Since we have assumed that $1\leq p$, we can use Proposition \ref{Prop:1LooksLikeThis} to deduce that $n_b=1$ for all $b\in\s O_0$.  Thus we have the following characterization
        \[p\;=\;1+\sum_{a\in\s O\setminus\s O_0}n_aa\;=\;1+C\,.\]
        Now we can observe that
        \[1+C=p=p^2=1+2C+C^2\,.\]
        By subtracting like terms, we find that
        \[0=C+C^2\,,\]
        so $C=0$ and hence $p=1$.
    \end{proof}

    We now proceed to analyze those $\mathbb N\Set$ algebras that arise from fusion categories, and abstract their properties.  For any simple object $X$ in a semisimple category $\mathcal C$, we denote its isomorphism class by $[X]$, and we write $\O(\mathcal C)$ for the set of isomorphism classes of simple objects.  
    
    \begin{definition}\label{Def:NO(C)}
        The free commutative monoid of isomorphism classes of objects of a semisimple category $\mathcal C$ is the object $\S(\mathcal C)$ in $\mathbb N\Set$.  We define a map $\ob(\mathcal C)\to\S(\mathcal C)$ by the formula $X\mapsto[X]$ for simple objects, and we extend to all objects by the formula $[X\oplus Y]:=[X]+[Y]$ (this uses semisimplicity).

        If $\mathcal C$ is (multi)fusion, then we call $\S(\mathcal C)$ the (multi)fusion semiring of $\mathcal C$, because it admits an algebra structure $\S(\mathcal C)\otimes\S(\mathcal C)\to\S(\mathcal C)$ induced by the tensor product in $\mathcal C$
        
        \[[X]\cdot[Y]:=[X\otimes Y]\,.\]

        We also define an involution $x\mapsto\overline{x}$ by the formula $[X]\mapsto[X^*]$.  Semisimplicity combines with duality in fusion categories to imply that $(X^*)^*\cong X$ for every object $X$, and this justifies the claim that this is an involution.
    \end{definition}

    \begin{remark}
        Given a functor $F:\mathcal C\to\mathcal D$, we can define a morphism $\S(F):\S(\mathcal C)\to\S(\mathcal D)$ by the formula $\S(F)\big([X]\big):=\big[F(X)\big]$.  This allows us to think of $\S$ as a functor from the category $\text{ssCat}$ of semisimple categories and natural isomorphism classes of functors to the category $\mathbb N\Set$.

        Note that this construction is not monoidal, since $\S(\mathcal C)\otimes\S(\mathcal D)$ has simple elements $\O(\mathcal C)\times\O(\mathcal D)$, while the simple elements of $\S(\mathcal C\boxtimes\mathcal D)$ correspond to the simple summands of $[X\boxtimes Y]$, for $X$, $Y$ simple, see Example \ref{Eg:RealDeligneTensors}.  Despite this, there is a comparison map $\S(\mathcal C)\otimes\S(\mathcal D)\to\S(\mathcal C\boxtimes\mathcal D)$ that makes the functor $\S$ lax monoidal.
    \end{remark}
    
    \begin{definition}
        An (abstract) multifusion semiring is an algebra object $A$ in the category $\mathbb N\Set$ with $\rk(A)<\infty$, together with an involution $a\mapsto\overline{a}$ such that for any simple elements $a,b\in A$, the following are equivalent
        \begin{enumerate}
            \item there exists a unique simple element $c\in A$ such that $c\leq1\cap ab$,
            \item there exists a unique simple element $d\in A$ such that $d\leq1\cap ba$,
            \item $b=\overline{a}$.
        \end{enumerate}
        When $1$ is simple, $A$ is said to be a fusion semiring.  When $1$ is not simple, $A$ is said to be strictly multi(fusion).
    \end{definition}

    \begin{example}
        When $\mathcal C$ is (multi)fusion, the $\mathbb N\Set$ algebra $\S(\mathcal C)$ is a (multi)fusion semiring.  This justifies the terminology established in Definition \ref{Def:NO(C)}.
    \end{example}

    \begin{remark}
        Our definition is similar to the definition of a based ring given in \cite[Def 3.1.3]{MR3242743}, though we allow for multiple copies of summands of 1 inside of $a\overline{a}$ when $a$ is simple.  When $A=\S(\mathcal C)$, this has the effect of allowing for $\End(X)$ to have dimension greater than one, which we need for the case when $X$ is a nonsplit simple object.
    \end{remark}

\begin{definition}
    An $\mathbb N\Set$ algebra is said to be \emph{transitive} if for any simple elements $x$ and $y$, there exist simple elements $u$ and $v$ such that $y\leq ux$ and $y\leq xv$.
\end{definition}

The next few propositions can be found in Chapter 3 of \cite{MR3242743}.  We will follow their development, and in the locations where nontrivial adaptations are required, we will give new proofs.

\begin{proposition}[{\cite[cf. Exercise 3.3.2]{MR3242743}}]\label{FusSemTrans}
    For $\mathcal C$ fusion, the fusion semiring $\S(\mathcal C)$ is transitive.
\end{proposition}

\begin{theorem}[Frobenius-Perron, e.g. \cite{MR3242743} Thm 3.2.1]\label{FPClassic}
    Any $n\times n$ matrix $M$ over $\mathbb R$ with non-negative entries has a non-negative eigenvalue $\lambda_{FP}$, which is larger in absolute value than all other eigenvalues of $M$.  If $M$ has an eigenvector with strictly positive entries, then the corresponding eigenvalue is $\lambda_{FP}$.
    
    If all entries of $M$ are strictly positive, then $\lambda_{FP}$ is positive, and has geometric multiplicity 1.  This unique eigenvector can be normalized so that all entries are positive.
\end{theorem}

This theorem is a consequence of the Brauer Fixed Point Theorem and elementary linear algebra.  As is often the case with proofs using elementary techniques, it is rather involved.  We encourage any readers unfamiliar with this result to consult \cite[Thm 3.2.1]{MR3242743}, because the proof is digestible and this theorem is the technical underpinning of this entire section.

\begin{definition}\label{DefFPdims}
    Let $A$ be a transitive $\mathbb N\Set$ algebra of finite rank.  There is a canonical $\mathbb N$-basis for $A$ given by the simple elements.  For any element $a\in A$, let $L_a$ and $R_a$ be the matrices of left and right multiplication by $a$ respectively, expressed in the basis of simple elements.  Define $\FPdim(a)$ to be the maximal non-negative eigenvalue of the matrix $L_a$ whose existence is guaranteed by the Frobenius-Perron theorem.
    When $X$ is an object of a fusion category, we will simply write $\FPdim(X)$ instead of $\FPdim([X])$.
\end{definition}

\begin{remark}
    It is important to stress the fact that this invariant $\FPdim$ is not well-behaved when the $\mathbb N\Set$ algebra is not transitive, and so we choose not to define it.  In particular, we will never apply this notation to $\S(\mathcal C)$ where $\mathcal C$ is multifusion. 
\end{remark}

The following suite of results, Propositions \ref{Prop:FPdimGeq1}-\ref{Prop:FPdimInvertible1}, are translated directly from \cite{MR3242743} into our setting.  We will need these results to prove Theorem \ref{Thm:NewRegular}, but the proofs of these proposition given \emph{loc. cit.} hold with only superficial modifications, so we do not reprove them here.

\begin{proposition}\cite[Prop 3.3.4]{MR3242743}\label{Prop:FPdimGeq1}
    For any $x\in\s O$,
    \begin{enumerate}
        \item $d=\FPdim(x)$ is an algebraic integer, and for any algebraic conjugate $d'$ of $d$, $d\geq|d'|$.
        \item $\FPdim(x)\geq1$.
    \end{enumerate}
\end{proposition}

\begin{proposition}\cite[Prop 3.3.6]{MR3242743}\label{Prop:FPdimProps}
    The following hold for any transitive $\mathbb N\Set$ algebra $A$ of finite rank.
    \begin{enumerate}
        \item The function $\FPdim:A\to\mathbb R$ is a character, i.e. a semiring homomorphism $A\to\mathbb C$.
        \item There exists a unique, up to scaling, nonzero element $\s R_A\in\mathbb R\otimes_{\mathbb N}A$ known as the \emph{regular element (of $A$)} such that $x\s R_A=\FPdim(x)\s R_A$ for all $x\in A$.  After appropriate normalization, this element has positive coefficients, and thus $\FPdim(\s R_A)>0$.
        \item $\FPdim$ is the unique character $A\to\mathbb C$ which takes non-negative values on $\s O$, and these values are actually strictly positive.
        \item If $x\in A$ has non-negative coefficients with respect to the basis of A, then $\FPdim(x)$ is the largest non-negative eigenvalue $\lambda_{FP}$ of the matrix $L_x$ of left multiplication by $x$.
    \end{enumerate}
\end{proposition}

\begin{proposition}\cite[Prop 3.3.9]{MR3242743}\label{Prop:FPdimOfXDual}
    When $A$ is a fusion semiring, $\FPdim(x)=\FPdim(\overline{x})$, for any $x\in A$.
\end{proposition}

\begin{proposition}\cite[Cor 3.3.10]{MR3242743} \label{Prop:FPdimInvertible1}
    When $\mathcal C$ is a fusion category, $\FPdim(g)=1$ if and only if $g$ is an invertible object.
\end{proposition}

\begin{theorem}\label{Thm:NewRegular}
    Let $\mathcal C$ be fusion over $\mathbb K$ with endomorphism field $\mathbb E=\End(\1_{\mathcal C})$.  Up to a positive scalar factor, the regular element $\s R=\s R_{\mathbb N\mathcal O(\mathcal C)}$ is given by
    \[\s R\;=\;\sum_{X\in\s O(\mathcal C)}\frac{\FPdim(X)}{\dim_{\mathbb E}\End(X)}\cdot [X]\,.\]
    The left embedding $\lambda_X:\mathbb E\hookrightarrow\End(X)$ (see Definition \ref{Def:LREmbeddings}) is used to define the vector space dimension in the above formula.  When performing computations with regular elements, we will use the definite article, and refer to this preferred normalization as \textbf{the} regular element of $\mathbb N\mathcal O(\mathcal C)$.
\end{theorem}

In order to prove this, it will be necessary first to establish a lemma in the spirit of \cite[Prop 3.1.6]{MR3242743}.

\begin{lemma}\label{Lem:CyclicPermLemma}
    For any multifusion category $\mathcal C$, and any $x,y,z\in\s O_{\S(\mathcal C)}$, the function $(x,y,z)\mapsto\dim_{\mathbb E}\End(Z)\cdot N_{x,y}^{\overline{z}}$ satisfies the following relations:
    \begin{gather*}
        \dim_{\mathbb E}\End(Z)\cdot N_{x,y}^{\overline{z}}\;=\;\dim_{\mathbb E}\End(Y)\cdot N_{z,x}^{\overline{y}}\;=\;\dim_{\mathbb E}\End(X)\cdot N_{y,z}^{\overline{x}}\\
        \dim_{\mathbb E}\End(Z)\cdot N_{x,y}^{\overline{z}}\;=\;\dim_{\mathbb E}\End(Z)\cdot N_{\overline{y},\overline{x}}^{z}
    \end{gather*}
    
\end{lemma}

\begin{proof}
    For every simple object $W$ in $\mathcal C$, denote $[W]\in\s O_{\S(\mathcal C)}$ by the corresponding lower case letter, i.e. $[W]=w$.  Using adjoint rules for duals, we can compute
    \begin{align*}
        \dim_{\mathbb E}\End(Z)\cdot N_{x,y}^{\overline z}&=\dim_{\mathbb E}\Hom(Z,Z)\cdot N_{x,y}^{\overline z}\\
        &=\dim_{\mathbb E}\Hom(Z^*,Z^*)\cdot N_{x,y}^{\overline z}\\
        &=\dim_{\mathbb E}\Hom(N_{x,y}^{\overline z}\cdot Z^*,Z^*)\\
        % &=\dim_{\mathbb E}\Hom\bigg(\bigoplus\nolimits_{K\in\s O(\s C)}N_{x,y}^{k}\cdot K,Z^*\bigg)\\
        &=\dim_{\mathbb E}\Hom\big(X\otimes Y\,,\,Z^*\big)\\
        % &=\dim_{\mathbb E}\Hom\big(X\otimes Y\,,\,Z^*\otimes\1\big)\\
        &=\dim_{\mathbb E}\Hom\big(Z^{**}\otimes(X\otimes Y)\,,\,\1\big)\\
        % &=\dim_{\mathbb E}\Hom\big(Z\otimes(X\otimes Y)\,,\,\1\big)\\
        &=\dim_{\mathbb E}\Hom\big((Z\otimes X)\otimes Y\,,\,\1\big)\\
        &=\dim_{\mathbb E}\Hom\big(Z\otimes X\,,\,Y^*\big)\\
        &=\dim_{\mathbb E}\End(Y)\cdot N_{z,x}^{\overline{y}}\,.
    \end{align*}
    By repeating the cyclic permutation, this proves the first relation.  To establish the second relation, observe that
    \begin{align*}
        \dim_{\mathbb E}\End(Z)\cdot N_{x,y}^{\overline z}&=\dim_{\mathbb E}\Hom\big(X\otimes Y\,,\,Z^*\big)\\
        &=\dim_{\mathbb E}\Hom\big(Z^*\,,\,X\otimes Y\big)\\
        &=\dim_{\mathbb E}\Hom\big(Y^*\otimes X^*\,,\,Z^{**}\big)\\
        &=\dim_{\mathbb E}\Hom\big(Y^*\otimes X^*\,,\,Z\big)\\
        &=\dim_{\mathbb E}\End(Z)\cdot N_{\overline{y},\overline{x}}^{z}\,.\qedhere
    \end{align*}
\end{proof}

Armed with this new lemma, we are now ready to prove Theorem \ref{Thm:NewRegular}.

\begin{proof}
    From part (2) of Proposition \ref{Prop:FPdimProps}, uniqueness means that it will suffice to show that $$w\s R=\FPdim(w)\s R.$$  By repeatedly applying Lemma \ref{Lem:CyclicPermLemma} it follows that
    \begin{align*}
        w\s R&=w\cdot\sum_{x\in\s O}\frac{\FPdim(x)}{\dim_{\mathbb E}\End(X)}\cdot x\\
        &=\sum_{x\in\s O}\frac{\FPdim(x)}{\dim_{\mathbb E}\End(X)}\cdot wx\\
        &=\sum_{x,y\in\s O}\frac{\FPdim(x)}{\dim_{\mathbb E}\End(X)}\cdot N_{w,x}^{y}y\\
        &=\sum_{x,y\in\s O}\frac{\FPdim(x)}{\dim_{\mathbb E}\End(X)\dim_{\mathbb E}\End(Y)}\cdot \dim_{\mathbb E}\End(Y)\cdot N_{w,x}^{y}y\\
        &=\sum_{x,y\in\s O}\frac{\FPdim(x)}{\dim_{\mathbb E}\End(X)\dim_{\mathbb E}\End(Y)}\cdot \dim_{\mathbb E}\End(X)\cdot N_{\overline{y},w}^{\overline{x}}y\\
        &=\sum_{x,y\in\s O}\frac{\FPdim(x)}{\dim_{\mathbb E}\End(X)\dim_{\mathbb E}\End(Y)}\cdot \dim_{\mathbb E}\End(X)\cdot N_{\overline{w},y}^{x}y\\
    % \end{align*}
    % \begin{align*}
        &=\sum_{y\in\s O}\frac{\FPdim\left(\sum_{x\in\s O}N_{\overline{w},y}^{x}x\right)}{\dim_{\mathbb E}\End(Y)}\cdot y\\
        &=\sum_{y\in\s O}\frac{\FPdim\left(\overline{w}y\right)}{\dim_{\mathbb E}\End(Y)}\cdot y\\
        &=\sum_{y\in\s O}\frac{\FPdim\left(\overline{w}\right)\FPdim\left(y\right)}{\dim_{\mathbb E}\End(Y)}\cdot y\\
        &=\FPdim\left(\overline{w}\right)\sum_{y\in\s O}\frac{\FPdim\left(y\right)}{\dim_{\mathbb E}\End(Y)}\cdot y\\
        &=\FPdim\left(\overline{w}\right)\s R\\
        &\hspace{-3pt}\mathop{=}\limits^{\ref{Prop:FPdimOfXDual}}\FPdim\left(w\right)\s R\qedhere
    \end{align*}
\end{proof}

\begin{definition}
    The Frobenius-Perron dimension of a fusion category $\s C$ over $\mathbb K$ is the Frobenius-Perron dimension of the regular element of its Grothendieck semiring:
    \[\FPdim(\s C):=\FPdim\big(\s R_{\S(\s C)}\big)\;=\;\sum_{x\in \s O}\frac{\FPdim(x)^2}{\dim_{\mathbb E}\End(X)}\,.\]
\end{definition}

\begin{example}\label{Eg:SomeFPComputations}
    Here are some computations of $\FPdim(\mathcal C)$ for various fusion categories.
    \begin{enumerate}
        \item In the category $\Vec_{\mathbb C}$ of complex vector spaces, $\FPdim(\mathbb C)=1$, since this is the monoidal unit.  There is only this one simple object, so $\FPdim(\Vec_{\mathbb C})=1$.
        \item Following Example \ref{RepQ8}, the $\FPdim$ of the category $\Rep_{\mathbb R}(Q_8)$ is
    \[\FPdim\big(\Rep_{\mathbb R}(Q_8)\big)\;=\;\frac{1^2}{1}+\frac{1^2}{1}+\frac{1^2}{1}+\frac{1^2}{1}+\frac{4^2}{4}\;=\;8\;=\;|Q_8|\,.\]
        \item Let $\mathbb F_2$ be the Galois field of order two, and let $\s C=\Rep_{\mathbb F_2}(\mathbb Z/3\mathbb Z)$.  This category is fusion because the characteristic of the base field does not divide the order of the group.  There are two simple objects, $\1$ and $V$, which are the irreducible representations.
    There is a basis $\{u,v\}\subseteq V$ for which the action of the group generator $\langle t\rangle=\mathbb Z/3\mathbb Z$ is given by
    \begin{align*}
        t.u&=v\,,\\
        t.v&=-u-v\,.
    \end{align*}
    These formulas show that $t^2+t+1$ acts by zero on $V$.  Since the group is commutative, the algebra of equivariant endomorphisms is
    \[\cong\mathbb F_2[\mathbb Z/3\mathbb Z]/(t^2+t+1)\cong\mathbb F_2[t]/(t^2+t+1)\cong\mathbb F_4\,.\]
    The fusion rules are given by $V\otimes V\cong \1\oplus\1\oplus V$.  Thus the characteristic polynomial of $L_V$ is $(\lambda-2)(\lambda+1)$, so by Proposition \ref{Prop:FPdimProps} part 4, this implies that $\FPdim(V)=2$.  With all of this data in hand, we can now compute $\FPdim(\s C)$ using Theorem \ref{Thm:NewRegular}:
    \[\FPdim(\s C)\;=\;\sum_{x\in\s O}\frac{\FPdim(x)^2}{\dim_{\mathbb E}\End(X)}\;=\;\frac{1^2}{1}\;+\;\frac{2^2}{2}\;=\;3\,.\]
    \end{enumerate}
\end{example}

By construction, $\FPdim(\s C)$ is an algebraic number.  In the AC case, the denominators in Theorem \ref{Thm:NewRegular} are all $1$, so the Frobenius-Perron dimensions of categories are actually algebraic \emph{integers}.  The above examples show that it is possible for these denominators to be greater than $1$, but that sometimes the numerator is divisible by this dimension, and the end result is an algebraic integer anyhow.  It turns out that this actually always happens, but it will take a few steps to prove this.  The argument is that there is a certain canonical object called the regular object, whose $\FPdim$ is exactly $\FPdim(\s C)$.  Such an object exists in the classical case as well, but analyzing its $\FPdim$ takes more work in the nAC setting.

\begin{definition}\label{DefDominant}
    A morphism $f:A\to B$ in $\mathbb N\Set$ is said to be dominant if for all $b\in B$, there is some $a\in A$ such that $b\leq f(a)$.  A tensor functor $(F,J):\s C\to\s D$ between fusion categories is said to be dominant if $\S(F):\S(\s C)\to \S(\s D)$ is dominant.
\end{definition}

\begin{remark}
    For a functor $F:\s C\to\s D$, being dominant is the same as the property that for every simple object $Y$ in $\s D$ there exists an object $X$ in $\s C$ such that $Y$ is a subobject of $F(X)$.
\end{remark}

\begin{proposition}\cite[Prop 3.3.13]{MR3242743}\label{fPlaysNiceWFPdims}
    Let $f:A\to B$ be morphism of $\mathbb N\Set$ algebras, and assume that $A$ and $B$ are transitive so that $\FPdim$s are defined.  In this setting, the following statements hold.
    \begin{enumerate}
        \item The map $f$ preserves Frobenius-Perron dimensions of objects.
        \item If $f$ is dominant, then
        \[f(\s R_A)\;=\;\frac{\FPdim(\s R_A)}{\FPdim(\s R_B)}\cdot\s R_B\,.\]
    \end{enumerate}
\end{proposition}

\begin{corollary}
    If a fusion category $\s C$ over $\mathbb K$ admits a tensor functor $(F,J):\s C\to \Vec_{\mathbb K}$, then $\FPdim(X)\in\mathbb N$ for any object $X$ in $\s C$.
\end{corollary}

In order to understand why $\FPdim$ is actually an algebraic integer, we will need a generalization of Proposition \ref{fPlaysNiceWFPdims}.

\begin{proposition}\label{fPlaysNiceButNowWN}
    Let $f:A\to B$ be a dominant morphism in $\mathbb N\Set$, and assume that $A$ and $B$ are transitive $\mathbb N\Set$ algebras so that $\FPdim$s are defined.  If there is some element $D\in A$ such that for any $x,y\in A$,
    \[f(x)f(y)\;=\;f(xDy)\,,\]
    then the following equations hold
    \begin{gather}
        \FPdim\big(f(x)\big)\;=\;\FPdim(D)\cdot\FPdim(x)\,,\\
        f(\s R_A)\;=\;\FPdim(D)\cdot\frac{\FPdim(\s R_A)}{\FPdim(\s R_B)}\cdot\s R_B\,.
    \end{gather}
    
\end{proposition}

\begin{proof}
    Set $M=\sum_{x\in\s O_A}f(x)$.  Consider the following computation:
    \begin{align*}
        Mf(\mathcal R_A)\;=\;\left(\sum_{x\in\s O_A}f(x)\right)f(\s R_A)&=\sum_{x\in\s O_A}f\left(xD\s R_A\right)\\
        &=\sum_{x\in\s O_A}f\left(\FPdim(xD)\s R_A\right)\\
        &=\FPdim(D)\cdot\left(\sum_{x\in\s O_A}\FPdim(x)\right)f(\s R_A)\,.
    \end{align*}
    For every simple element $b\in B$, $b\leq M$ by dominance of $f$.  The computation above shows that $f(\s R_A)$ is an eigenvector for left multiplication by $M$, and dominance implies that the coefficients of $f(\s R_A)$ are strictly positive.  Thus by Theorem \ref{FPClassic}, this eigenvector is unique up to scaling.  Since $\s R_B$ is evidently also such an eigenvector, there must be some scalar $r\in \mathbb R$ so that $f(\s R_A)=r\cdot\s R_B$.
    
    Using this proportionality scalar $r$, we can compute that for any $x\in A$,
    \begin{align*}
        \FPdim\big(f(x)\big)\s R_{B}&=f(x)\s R_B\\
        &=\tfrac1rf(x)f(\s R_A)\\
        &=\tfrac1rf(xD\s R_A)\\
        &=\tfrac1rf\big(\FPdim(xD)\s R_A\big)\\
        &=\FPdim(x)\FPdim(D)\tfrac1rf\big(\s R_A\big)\\
        &=\FPdim(x)\FPdim(D)\s R_B\,.
    \end{align*}
    By comparing the coefficients of $1_B$ in the above, we deduce that
    \[\FPdim\big(f(x)\big)\;=\;\FPdim(D)\cdot\FPdim(x)\,,\]
    which proves the first claim.  The second claim now follows from taking the $\FPdim$ of both sides of $f(\s R_A)=r\cdot\s R_B$ and applying part 1.
\end{proof}

\begin{example}\label{FreeBimoduleEg}
    Let $D$ be an indecomposable separable algebra object in a fusion category $\s C$ over $\mathbb K$.  Since $D$ is separable, the category of bimodules $\,_D\s C_D$ is also fusion.  For such an algebra, the free bimodule functor
    \begin{align*}
        F:\s C&\longrightarrow\;_D\s C_D\;,\\
        X&\longmapsto D\otimes X\otimes D\;,
    \end{align*}
    satisfies the property that
    \[FX\otimes_{D}FY\cong F(X\otimes D\otimes Y)\,.\]
    This implies that $f:=\S(F)$ satisfies the hypotheses of Proposition \ref{fPlaysNiceButNowWN}.
    
    This example is the exact situation that motivated the modification from Proposition \ref{fPlaysNiceWFPdims} to Proposition \ref{fPlaysNiceButNowWN}.  It should be noted that when working in this setting, the notation $\FPdim(D)$ is ambiguous.  Clearly $D$ is the monoidal unit in the category of $D$-bimodules, and hence $\FPdim(D)=1$ when we think of $D$ as a bimodule.  Thus we should be careful in this setting to write $\FPdim_{\,_D\s C_D}(D)$ or $\FPdim_{\s C}(D)$ as appropriate.
\end{example}

From Proposition \ref{fPlaysNiceButNowWN}, we can derive an important formula for computing the effect of adjoints on $\FPdim$s.  This next proposition is an analogue of \cite[Lemma 6.2.4]{MR3242743}.  The difference between our proposition and their lemma is that we only deal with the fusion case, but we allow for both $D\neq\1$ and a non-algebraically closed base field.

\begin{proposition}\label{Prop:FPdimsOfAdjoints}
    Let $A=\S(\s A)$ and $B=\S(\s B)$ be fusion semirings, and suppose $F:\s A\to\s B$ is a dominant functor such that $\S(F)=f$ satisfies the hypotheses of Proposition \ref{fPlaysNiceButNowWN} for some object $D$ in $\s A$.  Suppose further that $I:\s B\to \s A$ is a left or a right adjoint to $F$.  In this setting, the following formula holds for any $X$ in $\s B$
    \[\FPdim\big(I(X)\big)\;=\;\FPdim(D)\cdot\frac{d_{\s B}}{d_{\s A}}\cdot\frac{\FPdim(\s A)}{\FPdim(\s B)}\cdot\FPdim(X)\,,\]
    where $d_{\s A}$ and $d_{\s B}$ are the endomorphism degrees of $\s A$ and $\s B$ respectively.
\end{proposition}

First we develop some notation, then proceed to the proof.

\begin{definition}
    Let $\s C$ be a multifusion category over $\mathbb K$.  For any two objects $X$ and $Y$ in $\s C$, the pairing
    \[(X,Y)\longmapsto\dim_{\mathbb K}\Hom(X,Y)\]
    is constant on isomorphism classes, and sends direct sums in either argument to addition in $\mathbb N$.  It follows that $\dim_{\mathbb K}\Hom$ descends to an $\mathbb N$-bilinear pairing on the $\mathbb N\Set$ algebra $\S(\s C)$.  This pairing further induces a well-defined $\mathbb R$-bilinear pairing $\langle-\,,\,-\rangle$ on the real Grothendieck ring $\mathbb R\otimes_{\mathbb N}\S(\s C)$ by the following formula
    \[\left\langle\sum_Xr_X[X]\,,\,\sum_Yr_Y[Y]\right\rangle:=\sum_{X,Y}r_Xr_Y\dim_{\mathbb K}\Hom(X,Y)\,.\]
\end{definition}

\begin{proof}[Proof of Proposition \ref{Prop:FPdimsOfAdjoints}]
    Suppose that $I:\s B\to\s A$ is left adjoint to $F$ and that $X$ is a simple object of $\s B$.  We can compute
    \begin{align*}
        \FPdim(I(X))&=
        % \FPdim\left(\bigoplus_{V\in\s O_{\s A}}V\cdot\frac{\dim_{\mathbb E(\s A)}\Hom(I(X),V)}{\dim_{\mathbb E(\s A)}\Hom(V,V)}\right)\\
        \sum_{V\in\s O_{\s A}}\FPdim\left(V\right)\cdot\frac{\dim_{\mathbb E(\s A)}\Hom(I(X),V)}{\epsilon_V}\\
        % &=\sum_{V\in\s O_{\s A}}\frac{\FPdim\left(V\right)}{\epsilon_V}\cdot\dim_{\mathbb E(\s A)}\Hom(I(X),V)\\
        &=\sum_{V\in\s O_{\s A}}\frac{\FPdim\left(V\right)}{\epsilon_V}\cdot\frac{1}{d_{\s A}}\dim_{\mathbb K}\Hom(I(X),V)\\
        &=\sum_{V\in\s O_{\s A}}\frac{\FPdim\left(V\right)}{\epsilon_V}\cdot\frac{1}{d_{\s A}}\dim_{\mathbb K}\Hom(X,FV)\\
        &=\sum_{V\in\s O_{\s A}}\frac{\FPdim\left(V\right)}{\epsilon_V}\cdot\frac{1}{d_{\s A}}\big\langle[X]\,,\,f[V]\big\rangle\\
        &=\frac{1}{d_{\s A}}\left\langle[X]\,,\,f\left(\sum\limits_{V\in\s O_{\s A}}\frac{\FPdim\left(V\right)}{\epsilon_V}[V]\right)\right\rangle\\
        &=\frac{1}{d_{\s A}}\big\langle[X]\,,\,f\left(\s R_{\s B}\right)\big\rangle\\
        &\mathop{=}\limits^{\ref{fPlaysNiceButNowWN}}\frac{1}{d_{\s A}}\left\langle[X]\,,\,\FPdim(D)\cdot\frac{\FPdim(\s A)}{\FPdim(\s B)}\s R_{\s A}\right\rangle\\
        &=\frac{\FPdim(D)}{d_{\s A}}\cdot\frac{\FPdim(\s A)}{\FPdim(\s B)}\big\langle[X]\,,\,\s R_{\s A}\big\rangle\;=\;\cdots
    \end{align*}
    \begin{align*}
        &=\frac{\FPdim(D)}{d_{\s A}}\cdot\frac{\FPdim(\s A)}{\FPdim(\s B)}\left\langle[X]\,,\,\s \displaystyle\sum_{W\in\s O_{\s B}}\frac{\FPdim(W)}{\epsilon_W}[W]\right\rangle\\
        &=\frac{\FPdim(D)}{d_{\s A}}\cdot\frac{\FPdim(\s A)}{\FPdim(\s B)}\cdot\s \displaystyle\sum_{W\in\s O_{\s B}}\frac{\FPdim(W)}{\epsilon_W}\cdot\big\langle[X]\,,\,[W]\big\rangle\\
        % &=\frac{\FPdim(D)}{d_{\s A}}\cdot\frac{\FPdim(\s A)}{\FPdim(\s B)}\cdot\s \displaystyle\sum_{W\in\s O_{\s B}}\frac{\FPdim(W)}{\epsilon_W}\cdot\delta_{X,W}\cdot\dim_{\mathbb K}\End(X)\\
        &=\frac{\FPdim(D)}{d_{\s A}}\cdot\frac{\FPdim(\s A)}{\FPdim(\s B)}\cdot\frac{\FPdim(X)}{\dim_{\mathbb E}\End(X)}\cdot(d_{\s B}\cdot\dim_{\mathbb E}\End(X))\\
        &=\FPdim(D)\cdot\frac{d_{\s B}}{d_{\s A}}\cdot\frac{\FPdim(\s A)}{\FPdim(\s B)}\cdot\FPdim(X)\,.
    \end{align*}
    The fact that this holds for arbitrary $X$ follows from semisimplicity.  Any of the hom spaces in the above calculation could have had their arguments swapped without changing any of the dimensions, again by semisimplicity.  By swapping the direction of the homs, the above computation can be used to show the corresponding claim regarding right adjoints instead.
\end{proof}

In the AC setting, the $\FPdim$ of fusion categories is known to be a Morita invariant (see e.g. \cite{MR3242743}[Cor 7.16.7]).  In the nAC setting, the previous propositions point towards a suitable analogue.

\begin{theorem}\label{FPdimsNewMoritaInvariant}
    If $\s C$ and $\s D$ are Morita equivalent fusion categories over $\mathbb K$, then
    \[\frac{\FPdim(\s C)}{d_{\s C}}\;=\;\frac{\FPdim(\s D)}{d_{\s D}}\,.\]
\end{theorem}

\begin{proof}
    We start off by realizing the Morita equivalence between the two categories as $\s C_D$ for some separable algebra object $D$ in $\s C$.  This allows us to identify $\s D\simeq\;_D\s C_D$ as fusion categories.  The forgetful functor $U:\;_D\s C_D\to\s C$ is right adjoint to the free bimodule functor of Example \ref{FreeBimoduleEg}.  Thus, by Proposition \ref{Prop:FPdimsOfAdjoints}, we discover that for any $D$-bimodule $X$,
    \[\FPdim\big(U(X)\big)\;=\;\FPdim_{\s C}(D)\cdot\frac{d_{(\,_D\s C_D)}}{d_{\s C}}\cdot\frac{\FPdim(\s C)}{\FPdim(\,_D\s C_D)}\cdot\FPdim(X)\,.\]
    Setting $X=D$ the monoidal unit in the category of bimodules, the above formula produces
    \begin{gather*}
        \FPdim_{\s C}(D)\;=\;\FPdim_{\s C}(D)\cdot\frac{d_{\,_D\s C_D}}{d_{\s C}}\cdot\frac{\FPdim(\s C)}{\FPdim(\,_D\s C_D)}\cdot1\;\,.\\
    \end{gather*}
        After cancelling like terms and rearranging, we arrive at
    \begin{gather*}
        \frac{\FPdim(\s C)}{d_{\s C}}\;=\;\frac{\FPdim(\,_D\s C_D)}{d_{(\,_D\s C_D)}}\;=\;\frac{\FPdim(\s D)}{d_{\s D}}\,.\qedhere
    \end{gather*}
\end{proof}

\begin{corollary}\label{SimplifiedForgottenBimCor}
    Let $\s C$ be fusion and $D$ be an indecomposable separable algebra object in $\s C$, and let $X$ be a $D$-bimodule.  Then
    \[\FPdim\big(U(X)\big)\;=\;\FPdim_{\s C}(D)\cdot\FPdim(X)\,.\]
\end{corollary}

\begin{example}
    The category $(\mathbb C,\mathbb C)\Bim$ of Example \ref{Eg:CCbim} is evidently Morita equivalent to $\Vec_{\mathbb R}$.  For these small categories, Theorem \ref{FPdimsNewMoritaInvariant} is easily verified:
    \[\frac{\FPdim(\Vec_{\mathbb R})}{d_{\Vec_{\mathbb R}}}\;=\;\frac{1^2}{1}\;=\;\frac{1^2+1^2}{2}\;=\;\frac{\FPdim\big((\mathbb C,\mathbb C)\Bim\big)}{d_{(\mathbb C,\mathbb C)\Bim}}\,.\]
\end{example}

Another useful consequence of Proposition \ref{Prop:FPdimsOfAdjoints} is that it allows us to compute the $\FPdim$s of relative tensor products.

\begin{proposition}\label{FPdimsOfRelativeTensorProds}
    Let $D$ be an indecomposable, separable algebra in a fusion category $\s C$ over $\mathbb K$, and let $M$ and $N$ be right and left modules for $D$ respectively.  Then
    \[\FPdim(M\otimes_{D}N)\;=\;\frac{\FPdim(M)\cdot\FPdim(N)}{\FPdim(D)}\,.\]
\end{proposition}

\begin{proof}
    The hypotheses guarantee that $\,_D\s C_D$ is fusion, so that $\FPdim$s are well-defined.  As suggested in Example \ref{FreeBimoduleEg}, it will be convenient to introduce subscripts in order to tell bimodules from bare objects.  Let us write $\FPdim$ for objects in $\s C$, and $\FPdim_{bim}$ for $D$-bimodules.  We compute
    \begin{align*}
        \FPdim(M\otimes_DN)\cdot\FPdim(D)^2&=\FPdim\big((D\otimes M)\otimes_D(N\otimes D)\big)\\
        &\hspace{-5pt}\mathop{=}\limits^{\ref{SimplifiedForgottenBimCor}}\FPdim(D)\cdot\FPdim_{bim}\big((D\otimes M)\otimes_D(N\otimes D)\big)\\
        &=\FPdim(D)\cdot\FPdim_{bim}(D\otimes M)\cdot\FPdim_{bim}(N\otimes D)\\
        &\hspace{-5pt}\mathop{=}\limits^{\ref{SimplifiedForgottenBimCor}}\FPdim(D)\cdot\frac{\FPdim(D\otimes M)}{\FPdim(D)}\cdot\frac{\FPdim(N\otimes D)}{\FPdim(D)}\\
        &=\FPdim(D)\cdot\FPdim(M)\cdot\FPdim(N)\,.
    \end{align*}
    The claim then follows by dividing both sides by $\FPdim(D)^2$.
\end{proof}

\begin{example}
    Using the notation of Example \ref{NonNormalBim}, it follows from Proposition \ref{FPdimsOfRelativeTensorProds} that as objects in $\Vec_{\mathbb Q}$,
    \[\FPdim\big(\mathbb J(\omega)\otimes_{\mathbb J}\mathbb J(\omega)\big)=\frac{6\cdot 6}{3}\;=\;12\,.\]
    Using the Chinese remainder theorem, it can be verified that as $\mathbb J$-bimodules,
    \[\mathbb J(\omega)\otimes_{\mathbb J}\mathbb J(\omega)\cong\mathbb J\oplus\mathbb J\oplus\mathbb J(\omega)\;\mathop{\longmapsto}\limits^{\FPdim}\;3+3+6\;=\;12\;.\]
    Thus Proposition \ref{FPdimsOfRelativeTensorProds} can serve as a nice sanity check when performing such computations.
\end{example}

\begin{remark}
    The hypothesis that $D$ is separable is not superfluous.  Let $\s C=\Vec_{\mathbb K}$ and let $D$ be the algebra spanned by the vectors $1$ and $t$ subject to the relation $t^2=0$.  This algebra is not semisimple and therefore not separable.  Consider the right $D$-module $M=\langle m\rangle$ where $m\triangleleft t:=0$ and the left module $N=\langle n\rangle$ where $t\triangleright n:=0$.  The relative tensor product $M\otimes_{D}N$ is easily seen to be dimension 1 over $\mathbb K$.
    
    The author expects that there is a related, more complicated formula when the algebra is allowed to be nonsemisimple.
\end{remark}

As a more serious application of Proposition \ref{FPdimsOfRelativeTensorProds}, we turn to the task of showing that $\FPdim(\s C)$ is an algebraic integer.  The strategy is to show that there is an object $R$ in $\s C$ such that $\FPdim(R)=\FPdim(\s C)$, and then simply observe that the $\FPdim$s of any object are always algebraic integers by Proposition \ref{Prop:FPdimGeq1} part 1.

\begin{definition}
    Let $\s C$ be a fusion category over $\mathbb K$.  The canonical coend of $\s C$ is the object
    \[R:=\int^{V\in \s C}V^*\otimes V\cong\bigoplus_{X\in\s O(\s C)}\frac{X^*\otimes X}{\left\langle\big\{\,f^*\otimes\id_{X}-\id_{X^*}\otimes f\;|\;f:X\to X\,\big\}\right\rangle}\,.\]
\end{definition}

Given a simple object $X$ in $\s C$, the left embedding $\lambda_X:\mathbb E\hookrightarrow\End(X)$ shows that $\End(X)$ is an $\mathbb E$ algebra.  In other words, $\End(X)$ is an algebra object in the category $\Vec_{\mathbb E}=\langle\1\rangle\subseteq\s C$.  The identity on $\End(X)$ gives rise to a left $\End(X)$-module structure on $X$.  Similarly there is a right $\End(X)$-module structure on $X^*$.  These actions allow for another formulation of the canonical coend:

\[R\cong\bigoplus_{X\in\s O(\s C)}X^*\otimes_{\End(X)}X\,.\]

Using this formulation, we can prove the following.

\begin{theorem}\label{Thm:FPdimRIsFPdimC}
    The $\FPdim$ of a fusion category $\s C$ over $\mathbb K$ is given by the $\FPdim$ of its canonical coend $R$.  In particular, $\FPdim(\s C)$ is necessarily an algebraic integer.
\end{theorem}

\begin{proof}
    Simplicity of $X$ implies that $\End(X)$ is a division algebra, hence indecomposable.  Separability of $\s C$ implies that $\End(X)$ is separable.  Thus Proposition \ref{FPdimsOfRelativeTensorProds} applies, and we find that
    \begin{align*}
        \FPdim(R)&=\FPdim\left(\bigoplus_{X\in\s O(\s C)}X^*\otimes_{\End(X)}X\right)\\
        &=\sum_{X\in\s O(\s C)}\FPdim\left(X^*\otimes_{\End(X)}X\right)\\
        &=\sum_{X\in\s O(\s C)}\frac{\FPdim(X^*)\cdot\FPdim(X)}{\FPdim\big(\End(X)\big)}\\
        &=\sum_{X\in\s O(\s C)}\frac{\FPdim(X)^2}{\dim_{\mathbb E}\big(\End(X)\big)}\;=\;\FPdim(\s C)\qedhere
    \end{align*}
\end{proof}

\section{Consequences for the Drinfeld Center}\label{DrinfeldCenterSection}

The existence of Galois nontrivial objects, i.e. those $X$ where $\lambda_X\neq\rho_X$, in the nAC setting causes the Drinfeld center to behave very differently from the AC case.  By simply observing the way that elements of the endomorphism field $\mathbb E:=\End(\1)$ interact with the naturality of a half-braiding, we find that Galois nontrivial objects are `unbraidable', and that they restrict the endomorphism field of the center.  From these two primary observations a collection of interesting new phenomena arise.

Throughout this section, we assume that $\s C$ is a fusion category over $\mathbb K$.  Unfortunately, $\mathcal C$ being fusion does not imply that $\mathcal Z(\mathcal C)$ is fusion, as the following example shows.

\begin{example}
    Let $\mathcal C=\Vec_{\mathbb K}(G)$, and let $(V,\nu)$ be an object in $\Rep_{\mathbb K}(G)$ with $\dim_{\mathbb K}(V)=n$.  We can think of $V$ as the object $\1^{\oplus n}$ in $\mathcal C$, and on this object, we can define a half-braiding by the formula
    \[v\otimes t\longmapsto t\otimes \nu(g^{-1})v\,,\]
    whenever $v\in V$ and $t$ is homogeneous of degree $\deg(t)=g\in G$.  In this way, $\Rep_{\mathbb K}(G)$ can be seen to be a full subcategory of $\mathcal Z(\mathcal C)$.

    Regardless of the characteristic of the underlying field $\mathbb K$, $\mathcal C$ is separable.  However, $\Rep_{\mathbb K}(G)$ is separable if and only if $\char(\mathbb K)\nmid|G|$.  It follows that when $\char(\mathbb K)\mid|G|$, $\mathcal C$ is fusion, and yet $\mathcal Z(\mathcal C)$ is nonsemisimple.
\end{example}

The hypothesis that $\mathcal Z(\mathcal C)$ be fusion is necessary whenever computing Frobenius-Perron dimensions, and so we state this explicitly in Theorem \ref{Thm:FPdimOfZ(C)}.  We expect that a nonsemisimple generalization of nAC $\FPdim$s would make this assumption unnecessary, but this theory is beyond the scope of the present article.

\begin{proposition}[Unbraidability]\label{Unbraidable}
    Suppose $(Z,\gamma)$ is an object in $\s Z(\s C)$, and $X$ is a Galois nontrivial object in $\s C$.  It follows that $Z$ cannot have $X$ as a subobject, i.e. there can be no monomorphism $X\hookrightarrow Z$.
\end{proposition}

\begin{proof}
    Suppose for the sake of contradiction that $X$ was a subobject of $Z$.  Since $\s C$ is fusion, it follows that $Z\cong X\oplus Y$ for some other object $Y$.  Let $\id_{Z}=\iota_X\pi_X+\iota_Y\pi_Y$ be the corresponding decomposition of the identity map on $Z$.  By the naturality of the unit transformation $\ell$, it follows that for any $e\in\mathbb E$,
    \[\rho_Z(e)\;=\;\iota_X\rho_X(e)\pi_X+\iota_Y\rho_Y(e)\pi_Y\,.\]
    By the naturality of $\gamma$, we find that
    \[\lambda_Z(e)=\rho_Z(e)\,.\]
    Now by Galois nontriviality of $X$, there must exist some $t\in\mathbb E$ such that $\lambda_X(t)\neq\rho_X(t)$.  Thus, we find that
    \begin{align*}
        0&=\lambda_Z(t)-\rho_Z(t)\\
        &=\iota_X\lambda_X(t)\pi_X+\iota_Y\lambda_Y(t)\pi_Y-\iota_X\rho_X(t)\pi_X-\iota_Y\rho_Y(t)\pi_Y\,.
    \end{align*}
    By composing on the right with $\iota_X$, this becomes
    \begin{align*}
        0&=\iota_X\lambda_X(t)-\iota_X\rho_X(t)\\
        &=\iota_X\big(\lambda_X(t)-\rho_X(t)\big)\,.
    \end{align*}
    Since $\iota_X$ is monic, this forces a contradiction.
\end{proof}

\begin{corollary}\label{ForgFunctNotDominant}
    The forgetful functor $F:\s Z(\s C)\to \s C$ is not dominant whenever $\s C$ contains Galois nontrivial objects.
\end{corollary}

Thus Galois nontrivial objects are not in the image of the forgetful functor.  It turns out that the converse holds as well.  To see this, we will need to know a formula for the left adjoint to the forgetful functor.

\begin{theorem}\cite[cf. Thm 4.3]{dayCentres}\label{LeftAdjointFormula}
    At the level of objects, the left adjoint $I:\s C\to\s Z(\s C)$ to the forgetful functor $F:\s Z(\s C)\to\s C$ is given by the formula
    \[F\big(I(V)\big)\;=\;\int^{X\in\s C}X^*\otimes V\otimes X\,.\]
\end{theorem}

\begin{proof}
    The proof of this result is formal, elegant and involved.  Here we present a sketch of the argument, and the reader is encouraged to look at \cite[§3-§4]{bruguieresCategoricalCentersReshetikhinTuraev2008} and \cite[§6]{shimizuTheMonoidalCenter} for complete expositions, and also at \cite{loregian_2021} for basic coend manipulation techniques.
    
    Let us fix the notation
    \[T(V):=\int^{X\in \s C}X^*\otimes V\otimes X\,.\]
    This construction $T$ is functorial and comes equipped with canonical maps
    \[i(X,V):X^*\otimes V\otimes X\to T(V)\]
    that are natural in $V$ and dinatural in $X$.  The coend admits a monad structure $\mu:T^2\to T$ and $\eta:\id_{\s C}\to T$ which are uniquely determined on components by the formulae
    \begin{gather*}
        \mu_V\circ i(Y,TV)\circ\big(\id_{X^*}\otimes i(X,V)\otimes \id_X\big)\;=\;i(X\otimes Y,V)\,,\\
        \eta_V\;=\;i(\1,V)\,.
    \end{gather*}
    Observe that the data of an algebra $\gamma:T(Z)\to Z$ for this monad can be translated as follows
    \begin{align*}
        \Hom\big(T(Z),Z\big)&=\Hom\left(\int^{X}X^*\otimes Z\otimes X\,,\,Z\right)\\
        &=\int_{X}\Hom\left(X^*\otimes Z\otimes X\,,\,Z\right)\\
        &=\int_{X}\Hom\left(Z\otimes X\,,\,X\otimes Z\right)\\
        &=\text{Nat}\big(Z\otimes -\,,\,-\otimes Z\big)\,.
    \end{align*}
    This shows that an algebra map $\gamma:T(Z)\to Z$ can be turned into a candidate for a half-braiding on $Z$.  The hexagon axiom and invertibility of this half-braiding follow from the associativity and unitality of the algebra structure $\gamma$.  After verifying a few more details, it can be shown that the Eilenberg-Moore category $\s C^T$ of algebras for this monad $T:\s C\to\s C$ is equivalent to $\s Z(\s C)$ as a monoidal category.  Under this equivalence, the forgetful functor from $\s Z(\s C)\to \s C$ coincides with the forgetful functor from $\s C^T\to \s C$ which forgets the structure map.  The free $T$-algebra on an object $V$ in $\s C$ is precisely the algebra $\big(T(V),\mu_V\big)$, and this construction forms the left adjoint to the forgetful functor.  Thus after forgetting the structure map, we are left with the desired object.
\end{proof}

\begin{theorem}\label{Thm:ImageOfForgetful=GTs}
    The dominant image $\im(F)$ of the forgetful functor $F:\s Z(\s C)\to\s C$ consists of precisely the Galois trivial objects.
\end{theorem}

\begin{proof}
        Suppose $V$ is Galois trivial.  This means that for every $e:\1\to\1$, $\lambda_V(e)=\rho_V(e)$.  Since $e^*=e$, we can recast this Galois triviality statement as
        \[e^*\otimes \id_V\otimes \id_\1-\id_{\1^*}\otimes\id_V\otimes e\;=\;0\,.\]
        In other words the quotient map
        \[V\cong\1^*\otimes V\otimes\1\twoheadrightarrow\frac{\1^*\otimes V\otimes\1}{\left\langle\big\{ e^*\otimes \id_V\otimes \id_\1-\id_{\1^*}\otimes\id_V\otimes e\;\;|\;\;e:\1\to\1\,\big\}\;\right\rangle}\]
        is an isomorphism, since the set of relations is $\{0\}$.  This shows that
        \[V\hookrightarrow\bigoplus_{X\in\s O(\s C)}\frac{X^*\otimes V\otimes X}{\langle \textit{relations}\rangle}\;\cong\;T(V)\;=\;F\big(I(V)\big)\,,\]
        and so $V$ is in $\im(F)$.
\end{proof}

\begin{remark}
    The dominant image of a tensor functor between fusion categories is always a (full) fusion subcategory.  Thus, Theorem \ref{Thm:ImageOfForgetful=GTs} implies that the full subcategory spanned by the Galois trivial objects forms a fusion subcategory of $\mathcal C$.  Alternatively, this fact could be deduced from Definition \ref{Def:GaloisNontrivial} directly, but the proof would be longer.
\end{remark}

\begin{proposition}\cite[cf. proof of Lemma 5.6]{kong2019semisimple}\label{Prop:EndOf1InZ(C)}
    The endomorphism field of $\s Z(\s C)$ is
    \[\mathbb E\big(\s Z(\s C)\big)\;=\;\{e\in\mathbb E(\s C)\;|\;\forall X\in\s O(\s C),\,\lambda_X(e)=\rho_X(e)\}\]
\end{proposition}

\begin{proof}
    The monoidal unit in the center is $(\1,r^{-1}\circ\ell)$.  The condition that $e:\1\to\1$ respects this trivial braiding is equivalent to the statement that $\lambda_X(e)=\rho_X(e)$ for every $X$.
\end{proof}

\begin{example}
    The category $(\mathbb C,\mathbb C)\Bim$ has Drinfeld center equivalent to $\Vec_{\mathbb R}$.  Here Proposition \ref{Prop:EndOf1InZ(C)} singles out those complex scalars that are fixed by complex conjugation, hence the unit in the center has $\End(\1)\cong\mathbb R$.  In this example, the result of Corollary \ref{ForgFunctNotDominant} is evident from the fact that the forgetful functor is monoidal, so it must send the unit to the unit, which leaves nothing to dominate the conjugating bimodule.
\end{example}

\begin{theorem}\label{Thm:FPdimOfZ(C)}
    Suppose both $\mathcal C$ and $\mathcal Z(C)$ are fusion over $\mathbb K$, and let $F:\mathcal Z(\mathcal C)\to\mathcal C$ be the forget the forgetful functor. The Frobenius-Perron dimension of $\s Z(\s C)$ satisfies
    \[\FPdim\big(\s Z(\s C)\big)\;=\;\left(\frac{d_{\s Z(\s C)}}{d_{\s C}}\right)\FPdim\Big(\im(F)\Big)\FPdim(\s C)\leq\FPdim(\mathcal C)^2\,.\]
    Moreover, equality holds if and only if all objects of $\mathcal C$ are Galois trivial.
\end{theorem}

\begin{proof}
    Observe that the restricted forgetful functor $F_{*}:\s Z(\s C)\to\im(F)$ is monoidal and dominant by construction.  Let $I_*:\im(F)\to\s Z(\s C)$ be the left adjoint to $F_*$.  The functor $I_*$ can be computed by simply including the object into $\s C$, then applying $I$.
    
    From this, we can calculate
    \begin{gather*}
        \FPdim(\s C)\mathop{=}\limits^{\ref{Thm:FPdimRIsFPdimC}}\FPdim(R)
        \mathop{=}\limits^{\ref{LeftAdjointFormula}}\FPdim(FI\1)
        \mathop{=}\limits^{\ref{fPlaysNiceWFPdims}}\,\FPdim(I\1)
        =\FPdim(I_*\1)\\
        \mathop{=}\limits^{\ref{Prop:FPdimsOfAdjoints}}\,\FPdim(\1)\cdot\left(\frac{d_{\im(F)}}{d_{\s Z(\s C)}}\right)\cdot\frac{\FPdim\big(\s Z(\s C)\big)}{\FPdim\big(\im(F)\big)}\cdot\FPdim(\1)\\
        \;\;=\left(\frac{d_{\im(F)}}{d_{\s Z(\s C)}}\right)\cdot\frac{\FPdim\big(\s Z(\s C)\big)}{\FPdim\big(\im(F)\big)}\,.
    \end{gather*}
    The first claim follows from simply observing that $d_{\im(F)}=d_{\s C}$ as they have the same monoidal unit.
    
    Since $\im(F)$ is a full subcategory of $\mathcal C$, it follows that $\FPdim\big(\im(F)\big)\leq\FPdim(\mathcal C)$.  By Proposition \ref{Prop:EndOf1InZ(C)}, $\mathbb E\big(\mathcal Z(\mathcal C)\big)\subseteq\mathbb E(\mathcal C)$, and therefore $\frac{d_{\mathcal Z(\mathcal C)}}{d_{\mathcal C}}\leq1$.  These facts together establish the desired inequality.

    When all objects are Galois trivial, $\im(F)=\mathcal C$ by Theorem \ref{Thm:ImageOfForgetful=GTs}, and $\mathbb E\big(\mathcal Z(\mathcal C)\big)=\mathbb E(\mathcal C)$ by Proposition \ref{Prop:EndOf1InZ(C)}.  This establishes the `if' part of the final statement.

    For the `only if' part, suppose equality holds.  By dividing by $\FPdim(\mathcal C)^2$, this would imply that
    \[\left(\frac{d_{\s Z(\s C)}}{d_{\s C}}\right)\left(\frac{\FPdim\Big(\im(F)\Big)}{\FPdim(\s C)}\right)=1\,.\]
    Since we have already shown that the two factors on the left-hand side are each less than or equal to 1, it follows that they are both equal to 1.  In particular, $\FPdim\Big(\im(F)\Big)=\FPdim(\s C)$, and this can only happen when $\im(F)=\mathcal C$, so by Theorem \ref{Thm:ImageOfForgetful=GTs} all objects must be Galois trivial.
\end{proof}

\begin{example}
    Let $\mathbb L:=\mathbb Q(\zeta_7)$, where $\zeta_7$ is a primitive 7\textsuperscript{th} root of unity.  The Galois group of this extension is generated by the automorphism $\sigma:\zeta_7\mapsto\zeta_7^3$.  Consider the fusion category $\s C$ over $\mathbb Q$ whose simple objects are the $(\mathbb L,\mathbb L)$-bimodules $\mathbb L$, $\mathbb L_{\sigma^2}$ and $\mathbb L_{\sigma^4}$.  The fixed field of $\langle \sigma^2\rangle$ is the quadratic field $\mathbb Q(\sqrt{-7})$, and so $\s Z(\s C)\simeq\Vec_{\mathbb Q(\sqrt{-7})}$ and $d_{\s Z(\s C)}=2$.  We find that $\im(F)=\Vec_{\mathbb L}$, and so Theorem \ref{Thm:FPdimOfZ(C)} verifies that
    \[\FPdim\Big(\mathcal Z(\mathcal C)\Big)\;=\;1^2\;=\;\left(\frac{2}{6}\right)\cdot(1^2)\cdot(1^2+1^2+1^2)\;=\;\left(\frac{d_{\s Z(\s C)}}{d_{\s C}}\right)\FPdim\Big(\im(F)\Big)\FPdim(\s C)\,.\]
\end{example}

\appendix

\section{Separability Concerns}\label{Sec:Appendix}

The notion of a separable algebra is an important concept in the theory of associative algebras.  In the article \cite{ostrikModuleCatsWeakHopf}, Ostrik used the notion of separability of algebra objects as a stronger replacement for the property of semisimplicity.  This appendix addresses some, but not all of the issues regarding separability vs. semisimplicity.

Though the proof of Theorem \ref{SeparabilityImplications} is to our knowledge new, this appendix is primarily a collection of known results.  We encourage interested readers to look at \cite{kong2019semisimple} for a more in-depth investigation into separable algebras in multifusion categories.

\begin{definition}
    A $\mathbb K$-linear abelian category $\s C$ is said to be semisimple if all objects are projective.
\end{definition}

Semisimplicity is easy to define, and classically it is a component of the definition of a fusion category.  When allowing for non\textendash algebraically closed base fields, it turns out that semisimplicity is not well-behaved.  Specifically, there is a monoidal structure $\boxtimes$ on the 2-category of finite abelian $\mathbb K$-linear categories, and $\boxtimes$ does not preserve semisimplicity.  The next few definitions and examples clarify this inadequacy of semisimplicity, and serve as motivations for insisting that fusion categories be separable instead of just semisimple, as in Definition \ref{Def:mFus}.

\begin{definition}\cite[cf. Section 2.4]{lopezFrancoTensorProducts}\label{Def:DeligneTensor}
    Given two finite abelian $\mathbb K$-linear categories $\s C$ and $\s D$, their na\" ive tensor product $\s C\otimes_{\mathbb K}\s D$ is the category described as follows.
    \begin{itemize}
        \item Objects are formal pairs $(X,Y)$, which we will denote as $X\otimes_{\mathbb K}Y$, where $X$ is an object in $\s C$ and $Y$ is an object in $\s D$.
        \item Morphisms $f:X\otimes_{\mathbb K}Y\to U\otimes_{\mathbb K} V$ are elements of $\Hom(X,U)\otimes_{\mathbb K}\Hom(Y,V)$.
        \item Composition is given on simple tensors by $(f\otimes g)\circ(h\otimes k):=(f\circ h)\otimes(g\circ k)$, and extended to all morphisms by requiring bilinearity.
    \end{itemize}
    The (Deligne) tensor product $\mathcal C\boxtimes_{\mathbb K}\mathcal D$ is defined to be the completion of $\mathcal C\otimes_{\mathbb K}\mathcal D$ under finite colimits.
\end{definition}

\begin{example}\cite[cf. Section 2.1]{mugerFStCaTII}\label{Eg:SSDeligneDescription}
    Let $\mathcal C$ and $\mathcal D$ be two finite $\mathbb K$-linear abelian categories, and suppose that for any two objects $X\in\mathcal C$ and $Y\in\mathcal D$, the algebra $\End(X)\otimes_{\mathbb K}\End(Y)$ happens to be semisimple.  In this setting, the tensor product $\mathcal C\boxtimes_{\mathbb K}\mathcal D$ admits the following description:
    \begin{itemize}
        \item Objects are triples $(X,Y,e)$, where $X$ is an object in $\s C$ and $Y$ is an object in $\s D$, and where $e^2=e$ as an element of the algebra $\End(X)\otimes_{\mathbb K}\End(Y)$.
        \item Morphisms $f:(X,Y,e)\to(U,V,p)$ are elements of $\Hom(X,U)\otimes_{\mathbb K}\Hom(Y,V)$ that satisfy
        \[p\circ f\;=\;f\;=\;f\circ e\,.\]
        \item Composition is given on simple tensors by $(f\otimes g)\circ(h\otimes k):=\big((f\circ h)\otimes(g\circ k)\big)$, and extended to all morphisms by requiring bilinearity.
    \end{itemize}
    The object $(X,Y,\id_X\otimes\id_Y)$ is typically denoted $X\boxtimes Y$.
\end{example}

The above example demonstrates that, in the presence of semisimplicity, completion under finite colimits is equivalent to adding the images of all idempotents.  This construction is an example of something known as an idempotent completion (also known as Karoubi completion).  
This idea has been generalized to a notion of Karoubi completion for higher categories in \cite{gaiotto2019condensationshighercategories}
where idempotents are replaced with condensation monads.

\begin{remark}
    When simple objects $X$ in $\s C$ and $Y$ in $\s D$ are split, the algebra $\End(X)\otimes_{\mathbb K}\End(Y)\cong\mathbb K\otimes_{\mathbb K}\mathbb K\cong\mathbb K$, so the only nontrivial idempotent is $\id\otimes \id$.  This explains why the idempotent completion is not often discussed in the case where $\mathbb K$ is algebraically closed.
\end{remark}

Just as the tensor product of modules has an explicit construction as well as a definition via universal property, The tensor product $\boxtimes_{\mathbb K}$ can also be defined using a 2-universal property.  The `2' in 2-universal refers to the fact that this is a construction between objects in a 2-category, and thus this `definition' only specifies the category up to an equivalence, which is unique up to a unique 2-isomorphism.  Before stating the 2-universal property, we introduce some notation.

\begin{definition}
    Given any $n+1$ finite abelian $\mathbb K$-linear categories $\{\s C_i\}_{i=1}^n$ and $\s E$, there is a category $\Rex(\s C_1,\cdots,\s C_n;\s E)$.  The objects are functors from the cartesian product $\prod_{i}\s C_i\to \s E$ which are right-exact and $\mathbb K$-linear in each of the $\s C_i$, and the morphisms are natural transformations of such functors.
\end{definition}

\begin{proposition}\label{2Universal}
    The category $\s C\boxtimes_{\mathbb K} \s D$ satisfies the following universal property with respect to finite abelian categories over $\mathbb K$.  There exists a functor $Q_{\s C,\s D}:\s C\times\s D\to\s C\boxtimes_{\mathbb K}\s D$ such that, for any finite abelian $\mathbb K$-linear category $\s E$, restriction along this functor induces an equivalence
    \[\Rex(\s C\boxtimes_{\mathbb K}\s D,\s E)\mathop{\longrightarrow}\limits^{\simeq}\Rex(\s C,\s D;\s E)\,.\]
\end{proposition}

\begin{remark}\label{AltDefTensorProd}
    A common alternative definition of the operation $\s C\boxtimes_{\mathbb K}\s D$ is any pair $(\s C\boxtimes_{\mathbb K}\s D\,,\,Q_{\s C,\s D})$ that satisfies the 2-universal property of Proposition \ref{2Universal}.  Practitioners of algebra will be well aware of the usefulness of such an alternative characterization.  For example with $\otimes$, if one wishes to determine a map out of a tensor product, the universal property is what allows one to define the map on simple tensors.  However just as some universal objects don't exist in all categories, there are certain linear 2-categories where such a $\boxtimes$ fails to exist.  For more on the various settings in which such a universal category exists, see \cite{lopezFrancoTensorProducts}.
\end{remark}

Now let us explore some properties of the $\boxtimes_{\mathbb K}$ construction.

\begin{proposition}\cite[Example 11]{lopezFrancoTensorProducts}\label{ABMod}
    Let $A$ and $B$ be finite dimensional algebras over $\mathbb K$.  There is a canonical equivalence of categories
    \[A\Mod\boxtimes_{\mathbb K}B\Mod\simeq(A\otimes_{\mathbb K}B)\Mod\]
\end{proposition}

\begin{proof}
    Let $\s E$ be an additional finite abelian $\mathbb K$-linear category.  We can find an algebra $E$ such that $\s E\simeq E\Mod$.  This allows for the following computation.
    \begin{align*}
        \Rex(A\Mod\boxtimes_{\mathbb K}B\Mod,\s E)&\simeq\Rex(A\Mod\boxtimes_{\mathbb K}B\Mod,E\Mod)\\
        &\simeq\Rex(A\Mod,B\Mod;E\Mod)\\
        &\simeq\Rex\big(A\Mod,\Rex(B\Mod;E\Mod)\big)\\
        &\simeq\Rex\big(A\Mod,(E\otimes_{\mathbb K}B^{op})\Mod)\big)\\
        &\simeq\big((E\otimes_{\mathbb K}B^{op})\otimes_{\mathbb K}A^{op}\big)\Mod\\
        &\simeq\big(E\otimes_{\mathbb K}(B^{op}\otimes_{\mathbb K}A^{op})\big)\Mod\\
        &\simeq\big(E\otimes_{\mathbb K}(A\otimes_{\mathbb K}B)^{op}\big)\Mod\\
        &\simeq\Rex\big((A\otimes_{\mathbb K}B)\Mod,E\Mod\big)\\
        &\simeq\Rex\big((A\otimes_{\mathbb K}B)\Mod,\s E\big)\,.
    \end{align*}
    Here we have used the famous Eilenberg-Watts theorem to identify right exact functors with bimodules.  By plugging in $A\Mod\boxtimes_{\mathbb K}B\Mod$ in place of $\s E$, the above sequence produces the desired functor $F:(A\otimes_{\mathbb K}B)\Mod\to A\Mod\boxtimes_{\mathbb K}B\Mod$.  That $F$ is an equivalence follows from a version of the 2-Yoneda lemma, see e.g. \cite{johnson2DCats}.
\end{proof}

\begin{remark}
    For more on the Eilenberg-Watts theorem, see e.g. \cite{nymanAGeneralization} in the classical setting, or \cite{fuchsEWCalcRadford} for more modern applications.
\end{remark}

\begin{example}\label{Eg:SemXSem=NonSem}
    Let $\mathbb L$ be a (nontrivial) purely inseparable field extension of $\mathbb K$.  The category $\Vec_{\mathbb L}$ is semisimple, since all vector spaces are free and hence projective.  By Proposition \ref{ABMod}, the category $\s C:=\Vec_{\mathbb L}\boxtimes_{\mathbb K}\Vec_{\mathbb L}$ satisfies
    \[\Vec_{\mathbb L}\boxtimes_{\mathbb K}\Vec_{\mathbb L}\simeq(\mathbb L\otimes_{\mathbb K}\mathbb L)\Mod\,.\]
    Since $\mathbb L$ is purely inseparable, the algebra $(\mathbb L\otimes_{\mathbb K}\mathbb L)$ contains nontrivial nilpotent elements, say $x$.  The ideal $(x)$ is not projective as an $(\mathbb L\otimes_{\mathbb K}\mathbb L)$-module, and hence $\s C$ is non-semisimple.
\end{example}

The above example shows that semisimplicity is not preserved under $\boxtimes_{\mathbb K}$.  In order to establish a satisfactory theory for arbitrary fields, we would like to have a class of categories that is closed under $\boxtimes_{\mathbb K}$.  We focuses on the more restrictive class of \emph{separable} categories for this reason.  The definition of a separable category depends on the notion of a separable algebra, which we now define.

For the remainder of the appendix, we will occasionally use a more condensed notation for categories of (bi)modules.  For an algebra $A$ in a category $\mathcal C$, $_A\mathcal C=A\Mod$ is the category of left modules for $A$, $\mathcal C_A=\text{Mod-}A$ is the category of right modules, and $_A\mathcal C_A=(A,A)\Bim$ is the category of bimodules.

\begin{definition}
    An algebra $(A,\mu,\eta)$ is said to be separable if there exists a morphism $\theta:A\to A\otimes A$ in $_A\s C_A$, such that $\mu\circ\theta=\id_A$.
\end{definition}

\begin{example}
    For any field extension $\mathbb L/\mathbb K$, the field $\mathbb L$ is a separable algebra in $\Vec_{\mathbb K}
    $ if and only if $\mathbb L/\mathbb K$ is a separable extension.
\end{example}

\begin{example}
    Let $A=M_n(\mathbb K)$ be the $n\times n$ matrix algebra over $\mathbb K$, with $\mu(a,b)=a\cdot b$ being ordinary multiplication of matrices.  Then for any $1\leq j\leq n$, we can define the map
    \[\theta_j:A\longrightarrow A\otimes A\;,\;\;\;a\longmapsto\sum_{i=1}^na\cdot e_{i,j}\otimes e_{j,i}\,,\]
    where the $e_{i,j}$ are the elementary matrices.  From the multiplication rules in $A$, it is clear that $\mu\circ\theta_j=\id_{A}$.  We can then check that
    \begin{align*}
        \theta_j(e_{p,q})&=\sum_{i=1}^ne_{p,q}\cdot e_{i,j}\otimes e_{j,i}\;=\;\sum_{i=1}^n\delta_{q,i}\cdot e_{p,j}\otimes e_{j,i}\\
        &=e_{p,j}\otimes e_{j,q}\\
        &=\sum_{i=1}^ne_{i,j}\otimes e_{j,q}\cdot \delta_{p,i}\;=\;\sum_{i=1}^ne_{i,j}\otimes e_{j,i}\cdot e_{p,q}\;=\;\theta_j(1)\cdot e_{p,q}\,.
    \end{align*}
    Since the $e_{p,q}$ form a basis for $A$ as a $\mathbb K$-algebra, the above computation implies that
    \[a\cdot\theta_j(b)\;=\;\theta_j(ab)\;=\;\theta_j(a)\cdot b\,,\]
    or in other words $\theta_j$ is a morphism in $_A\mathcal C_A$.  Thus $A$ is a separable algebra.

\end{example}

\begin{example}\label{SeparableAW}
    More generally, it can be shown (see for e.g. \cite{yuanSeparableAlgebras2016}) that all finite dimensional separable algebras over $\mathbb K$ are direct sums of matrix algebras over division rings whose centers are separable extensions of $\mathbb K$.  Note the similarity between this classification result and the classical Artin-Wedderburn Theorem.
\end{example}

\begin{proposition}\label{AlgSepXSep=Sep}
    Given two separable algebras $A$ and $B$ in $\Vec_{\mathbb K}$, the tensor product $A\otimes B$ is also a separable algebra in $\Vec_{\mathbb K}$.
\end{proposition}

\begin{proof}
    Let $\sigma_{V,W}:V\otimes W\to W\otimes V$ be the swap map in $\Vec_{\mathbb K}$.  The algebra structure on $A\otimes B$ is defined to be
    \[\mu_{A\otimes B}:=(\mu_A\otimes\mu_B)\circ(\id_A\otimes\sigma_{B,A}\otimes \id_B)\,.\]
    By the separability of $A$ and $B$, we have splittings $\theta_A$ and $\theta_B$.  From the fact that $\sigma_{A,B}^{-1}=\sigma_{B,A}$, it follows that the composition
    \[\theta_{A\otimes B}:=(\id_A\otimes\sigma_{A,B}\otimes\id_B)\circ(\theta_A\otimes\theta_B)\]
    is a splitting for $\mu_{A\otimes B}$.  The bimodule conditions on $\theta_{A\otimes B}$ are easily verified, and so $A\otimes B$ is separable.
\end{proof}

\begin{definition}\label{Def:Separable}
    A finite abelian $\mathbb K$-linear category $\s C$ is said to be separable over $\mathbb K$ if there exists a separable algebra $A$ in $\Vec_{\mathbb K}$ such that $A\Mod\simeq\s C$.  Often we will simply say that a category is separable when the base field $\mathbb K$ can be inferred.
\end{definition}

\begin{remark}
    Separable categories are finite abelian $\mathbb K$-linear by definition.  More general notions of separability certainly warrant investigation, but we choose not to consider them here.
\end{remark}

The primary value of this class of category comes from the following property.

\begin{proposition}\label{CatSepXSep=Sep}
    Given two separable categories $\s C$ and $\s D$ over $\mathbb K$, the product $\s C\boxtimes_{\mathbb K}\s D$ is also separable over $\mathbb K$.
\end{proposition}

\begin{proof}
    Combine Example \ref{ABMod} with Proposition \ref{AlgSepXSep=Sep}.
\end{proof}

\begin{lemma}\cite[cf. Prop 4.2]{kong2019semisimple}\label{CatsOfMods4ASepAlgAreSS}
    For any separable algebra $A$ in a finite semisimple $\mathbb K$-linear abelian monoidal category $\s C$, the categories $\s C_A$, $_A\s C$ and $_A\s C_A$ are semisimple.
\end{lemma}

\begin{proof}
    By separability of $A$, choose some bimodule splitting $\theta:A\to A\otimes A$ of the multiplication map.  For any $A$-module $M$, the map $\theta$ induces a splitting $\theta_M:M\hookrightarrow A\otimes M$ of the module multiplication $\mu_M$ as follows.
    % https://q.uiver.app/?q=WzAsNixbMCwwLCJNIl0sWzEsMCwiXFwxXFxvdGltZXMgTSJdLFswLDEsIkFcXG90aW1lcyBNIl0sWzIsMCwiQVxcb3RpbWVzIE0iXSxbMiwxLCIoQVxcb3RpbWVzIEEpXFxvdGltZXMgTSJdLFsxLDEsIkFcXG90aW1lcyAoQVxcb3RpbWVzIE0pIl0sWzAsMSwiXFxjb25nIl0sWzEsMywiXFxldGEiLDAseyJzdHlsZSI6eyJ0YWlsIjp7Im5hbWUiOiJob29rIiwic2lkZSI6InRvcCJ9fX1dLFszLDQsIlxcdGhldGFcXG90aW1lc1xcaWRfTSIsMCx7InN0eWxlIjp7InRhaWwiOnsibmFtZSI6Imhvb2siLCJzaWRlIjoidG9wIn19fV0sWzQsNSwiXFxjb25nIiwyXSxbNSwyLCJcXGlkX0FcXG90aW1lc1xcbXVfTSIsMl0sWzAsMiwiXFx0aGV0YV9NIiwyLHsic3R5bGUiOnsiYm9keSI6eyJuYW1lIjoiZGFzaGVkIn19fV1d
    \[\begin{tikzcd}
    	M & {\1\otimes M} & {A\otimes M} \\
    	{A\otimes M} & {A\otimes (A\otimes M)} & {(A\otimes A)\otimes M\;.}
    	\arrow["\cong", from=1-1, to=1-2]
    	\arrow["\eta", hook, from=1-2, to=1-3]
    	\arrow["{\theta\otimes\id_M}", hook, from=1-3, to=2-3]
    	\arrow["\cong"', from=2-3, to=2-2]
    	\arrow["{\id_A\otimes\mu_M}"', from=2-2, to=2-1]
    	\arrow["{\theta_M}"', dashed, from=1-1, to=2-1]
    \end{tikzcd}\]
    This inclusion realizes $M$ as a summand of the free $A$-module $A\otimes M$. Thus if we can show that $A\otimes M$ is projective as a module, then $M$ must be projective as a module and the whole category is semisimple, since $M$ was arbitrary.
    
    Consider the lifting diagram below on the left.
    \vspace{2pt}
    % https://q.uiver.app/?q=WzAsMTAsWzAsMiwiQVxcb3RpbWVzIE0iXSxbMSwxLCJWIl0sWzEsMiwiVyJdLFsyLDBdLFsyLDNdLFswLDAsIl9BXFxtYXRoY2FsIEMiXSxbMywwLCJcXG1hdGhjYWwgQyJdLFszLDIsIk0iXSxbNCwxLCJWIl0sWzQsMiwiVyJdLFsxLDIsIiIsMCx7InN0eWxlIjp7ImhlYWQiOnsibmFtZSI6ImVwaSJ9fX1dLFswLDJdLFswLDEsIiIsMCx7InN0eWxlIjp7ImJvZHkiOnsibmFtZSI6ImRhc2hlZCJ9fX1dLFszLDQsIiIsMCx7InN0eWxlIjp7ImJvZHkiOnsibmFtZSI6ImRvdHRlZCJ9fX1dLFs3LDgsIiIsMCx7InN0eWxlIjp7ImJvZHkiOnsibmFtZSI6ImRhc2hlZCJ9fX1dLFs3LDldLFs4LDksIiIsMCx7InN0eWxlIjp7ImhlYWQiOnsibmFtZSI6ImVwaSJ9fX1dXQ==
    \[\begin{tikzcd}
    	{_A\mathcal C} && {} & {\mathcal C} \\
    	& V &&& V \\
    	{A\otimes M} & W && M & W \\
    	&& {}
    	\arrow[two heads, from=2-2, to=3-2]
    	\arrow[from=3-1, to=3-2]
    	\arrow[dashed, from=3-1, to=2-2]
    	\arrow[dotted, no head, from=1-3, to=4-3]
    	\arrow[dashed, from=3-4, to=2-5]
    	\arrow[from=3-4, to=3-5]
    	\arrow[two heads, from=2-5, to=3-5]
    \end{tikzcd}\]
    By applying the free-forgetful adjunction, the diagram on the left becomes the diagram on the right.  Since we have assumed $\s C$ to be semisimple, the lifting problem on the right has a solution.  By applying the adjunction in the other direction, the solution on the left is obtained.  The arguments for the right module and bimodule cases are completely analogous.
\end{proof}

\begin{remark}
    Closer inspection of the lifting argument above shows that the free module $A\otimes P$ is projective whenever $P$ is projective.  The fact that free modules aren't necessarily projective is important to remember if one is working in categories that are not semisimple to begin with.  We will not focus on such categories here, but we hope to translate the theory of tensor (not necessarily semisimple) categories into the nAC context, and this observation will be useful for that purpose.
\end{remark}

\begin{corollary}\label{Sep=>SS}
    A separable category is semisimple.
\end{corollary}

\begin{remark}\label{AlgFreeSepNote}
    The following is an equivalent `algebra free' definition of separability: Let $\s C$ be a finite $\mathbb K$-linear abelian category, and let $\mathrm{Rex}(\s C,\s C)$ be the category of right-exact $\mathbb K$-linear endofunctors (of $\s C$) with morphisms the natural transformations between such functors.  The category $\s C$ is said to be separable if the identity functor $\id_{\s C}:\s C\to\s C$ is projective as an object in $\mathrm{Rex}(\s C,\s C)$.
\end{remark}

\begin{theorem}\label{SeparabilityImplications}
    Let $\s C$ be a separable category, and let $A$ be an algebra object in of $\s C$.  If $A$ is separable, then the categories $_A\s C$, $\s C_A$ and $_A\s C_A$ are all separable.  If $\,_A\s C_A$ is semisimple, then $A$ is separable.
\end{theorem}

\begin{proof}
    To prove the first claim, suppose $A$ is separable.  By corollary \ref{CatsOfMods4ASepAlgAreSS}, we know that $_A\s C$ is semisimple.  Let $M$ be a simple $A$-module.  By the semisimplicity of $\s C$, when we forget the module structure, we can decompose $M$ into a direct sum of simple objects like $M\cong\bigoplus_iX_i$ in $\s C$.  We have seen in the proof of Lemma \ref{CatsOfMods4ASepAlgAreSS} that $M$ is a submodule of the object $A\otimes M$.  We can combine these to find that
    \[M\hookrightarrow A\otimes M\;\cong\;\bigoplus_iA\otimes X_i\,.\]
    By simplicity of $M$ as a module, it follows that $M$ can be embedded as a submodule of $A\otimes X_j$ for some simple object $X_j$ in $\s C$.  Thus all simple modules of $A$ appear as module summands of the module $A\otimes X:=A\otimes(\bigoplus_iX_i)$.  These simple summands correspond to projections in the finite dimensional algebra $\End_{_A\s C}(A\otimes X)$, and so there can only be finitely many of them.
    
    We have established that $_A\s C$ is a finite semisimple $\mathbb K$-linear abelian category, and so there is an equivalence $E\Mod\simeq\hspace{0pt} _A\s C$.  By Schur's lemma (\ref{Schur'sLemma}), the matrix blocks $M_n(D)$ of $E$ correspond to isomorphism classes of simple modules, and the classification in Example \ref{SeparableAW} shows that it will suffice to show that the underlying division ring $D$ of each of these matrix blocks is a separable algebra.  This division ring is of the form $D=\End_{_A\s C}(M)$ for some simple module $M$.
    
    Let us denote by $T=\End_{\s C}(M)$ the algebra of all endomorphisms of $M$ as an object in $\s C$.  The definition of morphisms in $_A\s C$ shows that we have an inclusion of algebras $D\hookrightarrow T$.  There is a special element $P\in \End_{\mathbb K}(T)$ defined by the following construction.
    \[P\Big(f:M\to M\Big):=\Big(M\hookrightarrow A\otimes M \mathop{\longrightarrow}\limits^{\id_A\otimes f} A\otimes M\mathop{\longrightarrow}\limits^{\mu_M}M\Big)\,.\]
    From this construction, it is clear that $P(f)\in D$ for any $f\in T$, and $P(f)=f$ when $f\in D$.  Slightly less obvious is the fact that
    \[\forall g\in T, \forall f,h\in D,\hspace{10pt} P(fgh)\;=\;fP(g)h\]
    In other words, if we think of $T$ as a $D$-bimodule, then $P$ determines a bimodule projection $T\twoheadrightarrow D$ that is a splitting of the inclusion $D\hookrightarrow T$.  Thus $D$ is a summand of the $D$-bimodule $T$.
    
    Since $D$ is a division algebra, $_DT\cong (_DD)^{\oplus m}$ as left modules (for some $m$).  By dimension arguments, $T_D\cong(D_D)^{\oplus m}$ as right modules with the same $m$, but possibly using a separate isomorphism.  Since we have assumed the category $\s C$ to be separable, we know that $T$ is separable as an algebra in $\Vec_{\mathbb K}$, so we have a bimodule inclusion
    \[D\hookrightarrow T\hookrightarrow T\otimes_{\mathbb K}T\cong (_DD)^m\otimes_{\mathbb K}(D_D)^m\cong(D\otimes_{\mathbb K}D)^{\oplus m^2}\,.\]
    This shows that $D$ is a summand of a  projective $D$-bimodule, and hence that $D$ itself is projective as a $D$-bimodule.  This is sufficient to prove that the multiplication map $D\otimes D\to D$ has a splitting as a bimodule map, so $D$ is separable.  The arguments for $\mathcal C_A$ and $_A\mathcal C_A$ are analogous.
    
    Now suppose that $\,_A\s C_A$ is semisimple.  This means that all objects are projective, so in particular $A$ is projective as an object in $\,_A\s C_A$.  Since the multiplication map is an epimorphism of bimodules, it must admit a splitting, and thus $A$ is separable.
\end{proof}

\clearpage
\newpage

\addcontentsline{toc}{section}{Bibliography}
\printbibliography

\end{document}